\documentclass[preprint,12pt]{elsarticle}
\usepackage[margin=1.6cm]{geometry}
\usepackage[all]{xy}
\usepackage[utf8x]{inputenc}
\usepackage[english]{babel}
\usepackage[T1]{fontenc}
\usepackage{enumitem}
\usepackage{stmaryrd,bbm}
\SetSymbolFont{stmry}{bold}{U}{stmry}{m}{n}
\usepackage{xspace}
\usepackage{comment}
\usepackage{url}
\usepackage{mathrsfs}
\usepackage{amsmath,amsfonts,amssymb,amsthm}

\usepackage{algorithm}
\usepackage{algorithmicx}
\usepackage{algpseudocode}
\usepackage{alltt}

\usepackage{tikz-cd}
\usetikzlibrary{decorations.markings}
\usetikzlibrary{decorations.pathmorphing}
\usetikzlibrary{patterns}

\newtheorem{thm}{Theorem}[section]
\newtheorem{prop}[thm]{Proposition}

\newtheorem{lemma}[thm]{Lemma}
\newtheorem{remark}{Remark}
\newtheorem{defi}[thm]{Definition}

\newcommand\coeff[1]{c_{#1}}
\newcommand\Down{\texttt{D}\xspace}
\newcommand\Up{\texttt{U}\xspace}
\newcommand\Left{\texttt{L}\xspace}
\newcommand\Right{\texttt{R}\xspace}
\newcommand\Sap{\textrm{SAP}}

\newcommand\cornerSAP{
\begin{tikzpicture}[x=2pt,y=2pt]
\draw(0,0)--(1,0)--(1,2)--(-1,2)--(-1,1)--(0,1)--(0,0);
\end{tikzpicture}}

\makeatletter
\def\ps@pprintTitle{%
  \let\@oddhead\@empty
  \let\@evenhead\@empty
  \let\@oddfoot\@empty
  \let\@evenfoot\@oddfoot
}
\makeatother

\begin{document}

\begin{frontmatter} 

\title{Fast construction of self-avoiding polygons and efficient evaluation of closed walk fractions on the square lattice}

\author[inst1]{Jean Fromentin}
\ead{jean.fromentin@univ-littoral.fr}
\author[inst1]{Pierre-Louis Giscard}
\ead{giscard@univ-littoral.fr}
\author[inst1]{Yohan Hosten}
\ead{yohan.hosten@math.cnrs.fr}
\affiliation[inst1]{organization={Universit\'e du Littoral C\^ote d'Opale, UR 2597, LMPA, Laboratoire de Math\'ematiques Pures et Appliqu\'ees Joseph Liouville},
addressline={50 rue F. Buisson},
postcode={F-62100},
city={Calais},
country={France}}

\begin{keyword}Self-avoiding polygon, computational combinatorics, loop-erasing, discrete version of the Schramm-Loewner Evolution, infinite resistor lattice\end{keyword}
	
\begin{abstract}
We build upon a recent theoretical breakthrough by employing novel algorithms to accurately compute the fractions $F_p$ of all closed walks on the infinite square lattice whose the last erased loop corresponds is any one of the $762, 207, 869, 373$ self-avoiding polygons $p$ of length at most 38. Prior to this work, only 6 values of $F_p$ had been calculated in the literature. The main computational engine uses efficient algorithms for both the construction of self-avoiding polygons and the precise evaluation of the lattice Green's function. Based on our results, we propose two conjectures: one regarding the asymptotic behavior of sums of $F_p$, and another concerning the value of $F_p$ when
$p$ is a large square. We provide strong theoretical arguments supporting the second conjecture. Furthermore, the algorithms we introduce are not limited to the square lattice and can, in principle, be extended to any vertex-transitive infinite lattice. In establishing this extension, we resolve two open questions related to the triangular lattice Green's function.
\end{abstract}

\end{frontmatter}

\section{Introduction}
In a series of seminal works on random generation of self-avoiding paths, G. Lawler introduced the procedure of chronological loop-erasing from a random walk \cite{Lawler1980, Lawler1999}. First conceived in the context of percolation theory to randomly generate simple paths--walks where all vertices are distinct--from a sample of random walks, the procedure consists of a chronological removal of cycles (called loops in Lawler's original work) as one walks along on the graph: consider for instance the complete graph on 4 vertices and label these with integers 1 through 4. Walking along the path $1\to 2\to 1 \to 3\to 4 \to 3\to 1\to 3$ on the graph and removing cycles whenever they appear, we are left with the simple path $1\to 3$  after having successively `erased' the cycles $1\to 2\to 1$, then  $3\to 4 \to 3$ and finally  $1\to 3 \to 1$. Note how $1\to 3\to 1$ does not appear contiguously in the original walk. Indeed, the edges traversed by the last erased loop might be very far apart from one another in the original walk.
Once terminated, Lawler's loop-erasing has eliminated a set of cycles, all of whose internal vertices are distinct, leaving a possibly trivial walk-skeleton behind. If the initial walk was itself a cycle, this skeleton is the empty walk on the initial vertex (also called length-0 walk) and otherwise it is a simple path (also called self-avoiding path). 
 This procedure yields a natural probability measure on self-avoiding paths: it is the probability of obtaining such a path as the self-avoiding skeleton left after erasing all loops from a random walk. For Brownian walks on infinite planar lattices under the continuum limit where edge lengths are sent to 0, this procedure is now known to produce a well-defined, conformally invariant probability distribution on self-avoiding paths called the Schramm-Loewner Evolution (SLE) with parameter $\kappa=2$ \cite{LawlerSLE2, Schramm2000}. Used on closed random walks (i.e. random walks with identical start and end points) the same procedure leads again to SLE$_2$ law on self-avoiding polygons. Here the self-avoiding polygon generated from the closed random walk is the very last loop eliminated by the procedure. While much effort has been expanded in the study of SLE, the equivalent discrete problems, that is on infinite lattices with \textit{finite edge size} have not been studied in nearly as much details and are comparatively less understood. 
 
 Consider an infinite, vertex transitive lattice with edges of unit size. The quantity that plays the role of SLE$_2$ in this discrete setting is the fraction $F_p$ of all closed walks whose last erased loop is a chosen self-avoiding polygon $p$. By definition, $0<F_p<1$. Furthermore, since the last erased loop of a walk is unique, the sum of $F_p$ over all self-avoiding polygons $p$ is 1,
 \begin{equation}\label{SumOne}
 \sum_{p:\,\text{SAP}} F_p=1,
 \end{equation}
 confirming that the $F_p$ have a probabilistic interpretation \cite{GISCARD2021}.
 Note that in this sum polygon orientation is retained and the starting point of a polygon is fixed so that, e.g., on the square lattice the $1\times 1$ square is counted 8 times.
 Beyond these basic assertions, very few results have been established concerning the $F_p$ numbers. In terms of explicit values, mappings to Abelian sand-pile models have permitted the calculation of $F_p$ for only the six shortest self-avoiding polygons on the infinite square lattice, owing to the mapping's complexity and ensuing lengthy computations \cite{Majumdar1991, Manna1992}. Regarding sums of $F_p$ values, it is not known how fast the sum of Eq.~(\ref{SumOne}) converges to 1 as the length $\ell$ of the self-avoiding polygons $p$ considered goes to infinity $\ell\to\infty$. Yet, one of the authors has shown in \cite{GISCARD2021} that this question is tied to the long open problem of determining the asymptotic growth of the number of self-avoiding polygons of length $\ell\to\infty$. 

In this work, we build on a recently derived analytical formula for the fractions $F_p$ \cite{GISCARD2021}, employing advanced computational techniques and novel algorithms to explicitly construct all $762,207,869,373$ self-avoiding polygons of length at most 38 on the infinite square lattice and evaluate their corresponding $F_p$. Additionally, we develop an algorithm to accurately compute individual $F_p$ values for self-avoiding polygons of length up to 1000. All source code is available for download \cite{GitHub}.
Thanks to these progresses we can conjecture that the sum in Eq.~(\ref{SumOne}) converges to 1 as follows,
$$
\sum_{\ell'\leq \ell} \sum_{p\,\in\, \text{SAP}_{\ell'}} F_p=1-\ell^{-3/5}+O(\ell^{-1}).
$$
where $\text{SAP}_{\ell'}$ is the set of self-avoiding polygons of length $\ell'$ and starting from a fixed vertex $O$ on the infinite square lattice. 
In addition, we can also conjecture that the polygon of length $4L$ with the largest $F_p$ value among all polygons of the same length is the $L\times L$ square $\text{Sq}_{L\times L}$. Furthermore we present solid, though incomplete, analytical arguments supporting the conjecture that 
$$
F_{\text{Sq}_{L\times L}} = \left(\sqrt{2}-1\right)^{4 L}+O(\text{polynom}(L)),
$$
which is also supported by numerical results. 
For the sake of concreteness we work on the square lattice, but our methods are generally valid on all vertex-transitive lattices. In the final section, we solve two open problems pertaining to the Green's function of the triangular lattice.

\subsection{Fraction $F_p$ of closed walks with last erased polygon $p$}
Let $G$ be an infinite vertex-transitive lattice, $\mathsf{A}_G$ its adjacency matrix and $\Lambda$ its unique dominant eigenvalue. Let $p$ be a self-avoiding polygon (SAP) or self-avoiding walk (SAW) on $G$ of length $\ell(p)$. Let $\mathcal{N}(p)$ be the distance-one neighborhood of $p$, that is, the set of vertices of $G$ that are at distance \textit{at most} one from any vertex visited by $p$ on $G$ (see Fig.~\ref{SAP}). Let $E_{\mathcal{N}(p)}$ be the set of all edges of $G$ with both ends in $\mathcal{N}(p)$. Let $\mathsf{B}_p$ be the adjacency matrix of the graph $G_{\mathcal{N}(p)}:=(\mathcal{N}(p),\,E_{\mathcal{N}(p)})$. This graph is finite provided $\ell(p)<\infty$ and so its adjacency matrix is of finite size $\mathcal{N}(p)\times \mathcal{N}(p)$. In addition, the size of this neighborhood cannot be too-large as compared to $\ell(p)$. Indeed, since $G$ is vertex-transitive it is regular and its dominant eigenvalue is the degree of all its vertices. This implies that $|\mathcal{N}(p)|\leq (\Lambda-1)\ell(p)$.

Let $\mathsf{R}(z):=(\mathsf{I}-z\mathsf{A}_G)^{-1}$ be the resolvent of $\mathsf{A}_G$ and let $\mathsf{P}_\Lambda$ be the projector onto the eigenspace of $\mathsf{A}_G$ associated with eigenvalue $\Lambda$. Define \cite{GISCARD2021}
 \begin{equation}\label{DefC}
 \mathsf{C} := \lim_{z\to1/\Lambda^{-}}~(\mathsf{I}-\mathsf{P}_\Lambda)\mathsf{R}(z).
 \end{equation} 
This matrix is directly related with the matrix $\mathsf{r}$ describing the resistance between pairs of points on the infinite resistor lattice with structure $G$ and with each edge representing a resistor of 1~Ohm. On infinite regular lattices we have the relation \cite{Klein1993} 
 \begin{equation}\label{CRrelation}
     \mathsf{C} = -\frac{1}{2}\mathsf{r}.
 \end{equation} 
 The value $\mathsf{C}_{O,v_i}$ between a vertex $v_i$ of the lattice and the origin is the difference of potentials between the origin and that vertex; while the resistance between these points is, by symmetry, half the potential difference between them. 
Research on the infinite resistor lattice problem \cite{Atkinson1999,Cserti2000,Cserti2011} together with Equation~(\ref{CRrelation}) mean that $\mathsf{C}$ is known for most vertex-transitive lattices. On the square lattice, $\mathsf{C}$  is explicitly given by
\begin{equation}
\mathsf{C}_{v_i,v_j}=-\frac{1}{\pi}\int_{0}^{\infty}\frac{1}{\tau}\left(1-\left(\frac{\tau-\mathbbm{i}}{\tau+\mathbbm{i}}\right)^{x_{ij}-y_{ij}}\left(\frac{\tau-1}{\tau+1}\right)^{x_{ij}+y_{ij}}\right)d\tau,\label{Cint}
\end{equation}
where $\mathbbm{i}^2=-1$, $x_{ij} = |x_{v_i} - x_{v_j}|$ and $y_{ij}=|y_{v_i} - y_{v_j}|$ are the distance along $x$ and $y$ between vertices $v_i$ and $v_j$, respectively. In this expression we have index entries of the matrix $\mathsf{C}$ by the lattice vertices $v_i$. While, in principle, this formula provides $\mathsf{C}$ explicitly, it is ill-suited to large scale numerical computations owing to the integration. We will present a method to bypass this issue in Section~\ref{Cmatrix}. 

Let $\mathsf{C}_p:=\mathsf{C}|_{\mathcal{N}(p)}$ be the restriction of $\mathsf{C}$ to the distance-one neighborhood of $p$. In other terms, $(\mathsf{C}_p)_{v_i,v_j}:=\mathsf{C}_{v_i,v_j}$ if and only if $v_i,v_j\in\mathcal{N}(p)$. The matrix $\mathsf{C}_p$ is, by definition, finite and of size $\mathcal{N}(p)\times \mathcal{N}(p)$.


\begin{thm}[\cite{GISCARD2021}]\label{InfiniteSieve}
Let $p$ be a self-avoiding polygon on $G$ and let $\mathsf{Id}$ be the $\mathcal{N}(p)\times \mathcal{N}(p)$ identity matrix. 
Let $\mathrm{deg}_p=\textrm{diag}(\mathsf{B}_p^2)$ be the column vector of vertex degrees on $G_{\mathcal{N}(p)}$ and let $\textbf{1}$ be the $\mathcal{N}(p)\times 1$ column vector full of 1.
 Then the asymptotic fraction  $0<F_p<1$ of all closed walks on the infinite lattice $G$ whose last erased loop is $p$ is well defined and given by
\begin{align}\label{Fpres}
     F_p &=\frac{1}{\Lambda^{\ell(p)}}\times \frac{1}{\Lambda}\,\mathrm{deg}_p^{T}\!.\,\mathrm{adj}\!\left(\mathsf{Id}+\frac{1}{\Lambda}\mathsf{C}_p.\mathsf{B}_p\right)\!.\textbf{1},     
\end{align}
where $\mathrm{adj}(\mathsf{M})$ designates the adjugate of a matrix $\mathsf{M}$. Furthermore $F_p\in\mathbb{Q}[\chi]$, with $\chi$ an irrational number depending on the lattice ($\chi=1/\pi$ for the square lattice).

 \end{thm}

\begin{remark}In spite of $G$ being infinite, Eq.~(\ref{Fpres}) yields a well defined number as both the $\mathsf{C}_p$ and $\mathsf{B}_p$ matrices are of finite size.
\end{remark}

\begin{remark}\label{Fphard}
By Theorem~\ref{InfiniteSieve}, each $F_p$ value is accessible only if both the matrix $\mathsf{C}$ and the self-avoiding polygon $p$ are known. In particular, we emphasize that since $\mathsf{B}_p$ is required to evaluate $F_p$, the self-avoiding polygon $p$ must be constructed explicitly. This represents far more information per polygon than had been made available even by the best enumerating techniques \cite{Conway_1993, Jensen_1999}.
\end{remark}

\begin{figure}
    \centering
    \begin{tikzpicture}
        \foreach \x in {-8,...,-4} {
            \draw (\x,0) -- (\x,4);
        }
        \foreach \y in {0,...,4} {
            \draw (-8,\y) -- (-4,\y);
        }
        \draw[color=red, line width=2pt] (-6,1)--(-5,1);
        \draw[color=red, line width=2pt] (-5,2)--(-5,1);
        \draw[color=red, line width=2pt] (-5,3)--(-5,2);       
        \draw[color=red, line width=2pt] (-5,3)--(-6,3);
        \draw[color=red, line width=2pt] (-7,3)--(-6,3);
        \draw[color=red, line width=2pt] (-7,3)--(-7,2);
        \draw[color=red, line width=2pt] (-6,2)--(-7,2);
        \draw[color=red, line width=2pt] (-6,2)--(-6,1);
        \draw[->, decorate, decoration={snake, amplitude=5pt, segment length=10pt, post length=5pt}, line width=0.75pt] (-3,2) -- (-1,2);
        \node[draw, circle, fill=black, inner sep=2pt, label={[font=\small]above left:$19$}] (A) at (0,2) {};
        \node[draw, circle, fill=black, inner sep=2pt, label={[font=\small]above left:$18$}] (B) at (0,3) {};
        \node[draw, circle, fill=black, inner sep=2pt, label={[font=\small]above left:$9$}] (C) at (1,1) {};
        \node[draw, circle, fill=red, inner sep=2pt, label={[font=\small]above left:\textcolor{red}{$7$}}] (D) at (1,2) {};
        \node[draw, circle, fill=red, inner sep=2pt, label={[font=\small]above left:\textcolor{red}{$6$}}] (E) at (1,3) {};
        \node[draw, circle, fill=black, inner sep=2pt, label={[font=\small]above left:$17$}] (F) at (1,4) {};
        \node[draw, circle, fill=black, inner sep=2pt, label={[font=\small]above left:$10$}] (G) at (2,0) {};
        \node[draw, circle, fill=red, inner sep=2pt, label={[font=\small]above left:\textcolor{red}{$1$}}] (H) at (2,1) {};
        \node[draw, circle, fill=red, inner sep=2pt, label={[font=\small]above left:\textcolor{red}{$8$}}] (I) at (2,2) {};
        \node[draw, circle, fill=red, inner sep=2pt, label={[font=\small]above left:\textcolor{red}{$5$}}] (J) at (2,3) {};
        \node[draw, circle, fill=black, inner sep=2pt, label={[font=\small]above left:$16$}] (K) at (2,4) {};
        \node[draw, circle, fill=black, inner sep=2pt, label={[font=\small]above left:$12$}] (L) at (3,0) {};
        \node[draw, circle, fill=red, inner sep=2pt, label={[font=\small]above left:\textcolor{red}{$2$}}] (M) at (3,1) {};
        \node[draw, circle, fill=red, inner sep=2pt, label={[font=\small]above left:\textcolor{red}{$3$}}] (N) at (3,2) {};
        \node[draw, circle, fill=red, inner sep=2pt, label={[font=\small]above left:\textcolor{red}{$4$}}] (O) at (3,3) {};
        \node[draw, circle, fill=black, inner sep=2pt, label={[font=\small]above left:$15$}] (P) at (3,4) {};
        \node[draw, circle, fill=black, inner sep=2pt, label={[font=\small]above left:$11$}] (Q) at (4,1) {};
        \node[draw, circle, fill=black, inner sep=2pt, label={[font=\small]above left:$13$}] (R) at (4,2) {};
        \node[draw, circle, fill=black, inner sep=2pt, label={[font=\small]above left:$14$}] (S) at (4,3) {};
        \draw (A)--(D);
        \draw (B)--(E);
        \draw (C)--(D);\draw (C)--(H);
        \draw[color=red, line width=2pt] (D)--(E);\draw[color=red, line width=2pt] (D)--(I);
        \draw (E)--(F);\draw[color=red, line width=2pt] (E)--(J);
        \draw (G)--(H);
        \draw[color=red, line width=2pt] (H)--(M);\draw[color=red, line width=2pt] (H)--(I);
        \draw (I)--(J);\draw (I)--(N);
        \draw (J)--(K);\draw[color=red, line width=2pt] (J)--(O);
        \draw (L)--(M);
        \draw[color=red, line width=2pt] (M)--(N);\draw (M)--(L);\draw (M)--(Q);
        \draw[color=red, line width=2pt] (N)--(O);\draw (N)--(R);
        \draw (O)--(P);\draw (O)--(S);
    \end{tikzpicture}
    \caption{A self-avoiding polygon on the square lattice and its distance-one neighbors. The matrix $\mathsf{B}_{\protect\cornerSAP}$ is the adjacency matrix of the graph on the right hand side.}
    \label{SAP}
\end{figure}
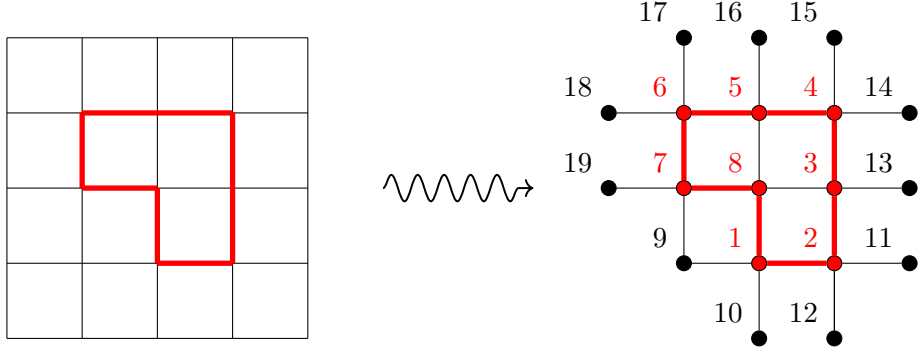
As an illustrative example of the Theorem, consider the corner-shape self-avoiding polygon $\cornerSAP$ shown in Figure~\ref{SAP}. We have $|\mathcal{N}(\cornerSAP)|=19$ so the sieving result providing the exact fraction $F_{\cornerSAP}$ involves the adjugate of a $19\times 19$ matrix. By construction $\mathsf{B}_{\cornerSAP}$ and $\mathsf{C}_{\cornerSAP}$ are symmetric and $\mathsf{C}_{\cornerSAP}$ further presents the same symmetries as the underlying square lattice $G$. This reduces the number of distinct entries of $\mathsf{C}_{p}$ that need to be computed. In the concrete case of $\cornerSAP$, labeling the vertices of $G_{\mathcal{N}(\cornerSAP)}$ as in Fig.~\ref{SAP} (right), the matrix $\mathsf{C}_{\cornerSAP}$ is given by
$$\mathsf{C}_{\cornerSAP}=\begin{pmatrix}
    \mathsf{C}_1&\mathsf{C}_2&\mathsf{C}_3\\ \mathsf{C}_2& \mathsf{C}_4&\mathsf{C}_5 \\\mathsf{C}_3& \mathsf{C}_5&\mathsf{C}_6
\end{pmatrix},$$ 
where the $\mathsf{C}_j$ matrices are 
\begin{itemize}
    \item $\mathsf{C}_1$ is the $9\times 9$ matrix given by
        \begin{equation*}
        \mathsf{C}_1=\begin{pmatrix}0 & -1 & -\frac{4}{\pi } & 1-\frac{8}{\pi } & \frac{8}{\pi }-4 & 1-\frac{8}{\pi } & -\frac{4}{\pi } & -1 & -1 \\ 
            -1 & 0 & -1 & \frac{8}{\pi }-4 & 1-\frac{8}{\pi } & -\frac{16}{3 \pi } & 1-\frac{8}{\pi } & -\frac{4}{\pi } & \frac{8}{\pi }-4 \\
            -\frac{4}{\pi } & -1 & 0 & -1 & -\frac{4}{\pi } & 1-\frac{8}{\pi } & \frac{8}{\pi }-4 & -1 & 1-\frac{8}{\pi}\\
            1-\frac{8}{\pi } & \frac{8}{\pi }-4 & -1 & 0 & -1 & \frac{8}{\pi }-4 & 1-\frac{8}{\pi } & -\frac{4}{\pi } & -\frac{16}{3 \pi }\\ 
            \frac{8}{\pi }-4 & 1-\frac{8}{\pi } & -\frac{4}{\pi } & -1 & 0 & -1 & -\frac{4}{\pi } & -1 & 1-\frac{8}{\pi } \\
             1-\frac{8}{\pi } & -\frac{16}{3 \pi } & 1-\frac{8}{\pi } & \frac{8}{\pi }-4 & -1 & 0 & -1 & -\frac{4}{\pi } & \frac{8}{\pi }-4 \\
             -\frac{4}{\pi } & 1-\frac{8}{\pi } & \frac{8}{\pi }-4 & 1-\frac{8}{\pi } & -\frac{4}{\pi } & -1 & 0 & -1 & -1  \\
             -1 & -\frac{4}{\pi } & -1 & -\frac{4}{\pi } & -1 & -\frac{4}{\pi } & -1 & 0 & -\frac{4}{\pi } \\
             -1 & \frac{8}{\pi }-4 & 1-\frac{8}{\pi } & -\frac{16}{3 \pi } & 1-\frac{8}{\pi } & \frac{8}{\pi }-4 & -1 & -\frac{4}{\pi } & 0 
            \end{pmatrix};
        \end{equation*}
    \item $\mathsf{C}_2$ is the $9\times 1$ vector given by
        \begin{equation*}
            \mathsf{C}_2^T=\begin{pmatrix}-1&-\frac{4}{\pi }&1-\frac{8}{\pi }&8-\frac{92}{3 \pi } &\frac{48}{\pi }-17&8-\frac{92}{3 \pi }&1-\frac{8}{\pi }&\frac{8}{\pi }-4&-\frac{4}{\pi }
                                \end{pmatrix};
        \end{equation*}
    \item $\mathsf{C}_3$ is the $9\times 9$ matrix 
        \begin{equation*}
            \mathsf{C}_3=\begin{pmatrix}
                \frac{8}{\pi }-4 & -\frac{4}{\pi } & 1-\frac{8}{\pi } & -\frac{16}{3 \pi } & 8-\frac{92}{3 \pi } & \frac{48}{\pi }-17 & 8-\frac{92}{3 \pi } & -\frac{16}{3 \pi } & 1-\frac{8}{\pi } \\
                -1 & -1 & -\frac{4}{\pi } & 1-\frac{8}{\pi } & \frac{48}{\pi }-17 & 8-\frac{92}{3 \pi } & -1-\frac{8}{3 \pi } & -1-\frac{8}{3 \pi } & 8-\frac{92}{3 \pi } \\
                -\frac{4}{\pi } & \frac{8}{\pi }-4 & -1 & -\frac{4}{\pi } & \frac{8}{\pi }-4 & 1-\frac{8}{\pi } & -\frac{16}{3 \pi } & 8-\frac{92}{3 \pi } & \frac{48}{\pi }-17 \\
                1-\frac{8}{\pi } & \frac{48}{\pi }-17 & -\frac{4}{\pi } & -1 & -1 & -\frac{4}{\pi } & 1-\frac{8}{\pi } & \frac{48}{\pi }-17 & 8-\frac{92}{3 \pi } \\
                -\frac{16}{3 \pi } & 8-\frac{92}{3 \pi } & 1-\frac{8}{\pi } & \frac{8}{\pi }-4 & -\frac{4}{\pi } & -1 & -\frac{4}{\pi } & \frac{8}{\pi }-4 & 1-\frac{8}{\pi } \\
                -1-\frac{8}{3 \pi } & -1-\frac{8}{3 \pi } & 8-\frac{92}{3 \pi } & \frac{48}{\pi }-17 & 1-\frac{8}{\pi } & -\frac{4}{\pi } & -1 & -1 & -\frac{4}{\pi } \\
                8-\frac{92}{3 \pi } & -\frac{16}{3 \pi } & \frac{48}{\pi }-17 & 8-\frac{92}{3 \pi } & -\frac{16}{3 \pi } & 1-\frac{8}{\pi } & \frac{8}{\pi }-4 & -\frac{4}{\pi } & -1 \\
                1-\frac{8}{\pi } & 1-\frac{8}{\pi } & \frac{8}{\pi }-4 & 1-\frac{8}{\pi } & 1-\frac{8}{\pi } & \frac{8}{\pi }-4 & 1-\frac{8}{\pi } & 1-\frac{8}{\pi } & \frac{8}{\pi }-4 \\
                \frac{48}{\pi }-17 & 1-\frac{8}{\pi } & 8-\frac{92}{3 \pi } & -1-\frac{8}{3 \pi } & -1-\frac{8}{3 \pi } & 8-\frac{92}{3 \pi } & \frac{48}{\pi }-17 & 1-\frac{8}{\pi } & -\frac{4}{\pi } \\
            \end{pmatrix};
        \end{equation*}
        \item $\mathsf{C}_4=0$;
        \item $\mathsf{C}_5$ is the $1\times 9$ vector
             \begin{equation*}
            \mathsf{C}_5=\begin{pmatrix}1-\frac{8}{\pi } & -1 & -\frac{16}{3 \pi } & -1-\frac{8}{3 \pi } & 49-\frac{160}{\pi } & \frac{736}{3 \pi }-80 & 49-\frac{160}{\pi } & -1-\frac{8}{3 \pi } & -\frac{16}{3 \pi }
                                \end{pmatrix};
        \end{equation*}
        \item $\mathsf{C}_6$ is the $9\times 9$ matrix
            \begin{equation*}
            \mathsf{C}_6=\begin{pmatrix}
                0 & -\frac{4}{\pi } & -1 & \frac{8}{\pi }-4 & 8-\frac{92}{3 \pi } & -1-\frac{8}{3 \pi } & -\frac{92}{15 \pi } & \frac{472}{15 \pi }-12 & 49-\frac{160}{\pi} \\
                -\frac{4}{\pi } & 0 & 1-\frac{8}{\pi } & 8-\frac{92}{3 \pi } & \frac{736}{3 \pi }-80 & 49-\frac{160}{\pi } & \frac{472}{15 \pi }-12 & -\frac{92}{15 \pi } & -1-\frac{8}{3 \pi }\\
                -1 & 1-\frac{8}{\pi } & 0 & -1 & 1-\frac{8}{\pi } & -\frac{16}{3 \pi } & -1-\frac{8}{3 \pi } & 49-\frac{160}{\pi } & \frac{736}{3 \pi }-80\\
                \frac{8}{\pi }-4 & 8-\frac{92}{3 \pi } & -1 & 0 & -\frac{4}{\pi } & 1-\frac{8}{\pi } & 8-\frac{92}{3 \pi } & \frac{736}{3 \pi }-80 & 49-\frac{160}{\pi } \\
                8-\frac{92}{3 \pi } & \frac{736}{3 \pi }-80 & 1-\frac{8}{\pi } & -\frac{4}{\pi } & 0 & -1 & \frac{8}{\pi }-4 & 8-\frac{92}{3 \pi } & -1-\frac{8}{3 \pi }\\
                -1-\frac{8}{3 \pi } & 49-\frac{160}{\pi } & -\frac{16}{3 \pi } & 1-\frac{8}{\pi } & -1 & 0 & -1 & 1-\frac{8}{\pi } & -\frac{16}{3 \pi } \\
                 -\frac{92}{15 \pi } & \frac{472}{15 \pi }-12 & -1-\frac{8}{3 \pi } & 8-\frac{92}{3 \pi } & \frac{8}{\pi }-4 & -1 & 0 & -\frac{4}{\pi } & 1-\frac{8}{\pi }\\
                 \frac{472}{15 \pi }-12 & -\frac{92}{15 \pi } & 49-\frac{160}{\pi } & \frac{736}{3 \pi }-80 & 8-\frac{92}{3 \pi } & 1-\frac{8}{\pi } & -\frac{4}{\pi } & 0 & -1 \\
                 49-\frac{160}{\pi } & -1-\frac{8}{3 \pi } & \frac{736}{3 \pi }-80 & 49-\frac{160}{\pi } & -1-\frac{8}{3 \pi } & -\frac{16}{3 \pi } & 1-\frac{8}{\pi } & -1 & 0 \\
            \end{pmatrix}.
            \end{equation*}
\end{itemize}

\noindent Now the fraction $F_{\cornerSAP}$ is given by Eq.~(\ref{Fpres})
as 
\[F_{\cornerSAP}=\frac{1}{576\, \pi ^6} (3 \pi -8)^2 (8-\pi) (4-\pi) (-23\pi^2+120\pi-128\big)\simeq 3.36\times 10^{-4},\]
that is, $0.0336\%$ of all closed walks on the infinite square lattice have $\cornerSAP$ as their last erased loop.

\subsection{Sums of fractions over families of self-avoiding polygons}\label{Sums}
Fix a vertex $O$ on the infinite square lattice and denote $\text{dSAP}_\ell$  the set of self-avoiding polygons of length $\ell$ from $O$ to itself. In this set, self-avoiding polygons that differ only through their orientations or through a translation (i.e. changing the starting point) are considered distinct.
Let $\Sap_\ell$ be the set of self-avoiding polygons of length $\ell$, up to orientation and translation. Let  $\pi(\ell)=|\Sap_\ell|$, then $|\text{dSAP}_\ell|=2\ell\times \pi(\ell)$ since a self-avoiding polygon has 2 orientations and $\ell$ possible starting points.
Let $F(\ell)$ be the sum of $F_p$ values for all $p\in\text{dSAP}_\ell$, that is
$$
F(\ell):=\sum_{p\,\in\,\text{dSAP}_\ell} F_p=2\ell \sum_{p\,\in\,\Sap_\ell} F_p.
$$
Of particular interest are the cumulative sums,
$$
S(\ell):=\sum_{\ell'\leq \ell} F_{\ell'}.
$$
$S(\ell)$ gives the proportions of closed walks on the infinite square lattice whose last erased loop is a self-avoiding polygon of length at most $\ell$. The behavior of $S(\ell)$ as $\ell\to \infty$ on the infinite discrete lattices is a determinant signature of the probability law on self-avoiding polygons resulting from loop-erasing. 
In practice, evaluating the sums $S(\ell)$ is a major computational challenge because they stem from hundreds of billions of individual $F_p$ values, each of which, as noted in Remark~\ref{Fphard}, requires accessing the matrix $\mathsf{C}$ and knowing the self-avoiding polygon $p$ in full details. In other terms, a numerical study of the behavior of $S(\ell)$ as $\ell$ grows requires explicitly constructing all self-avoiding polygons of length $\ell$ as well as an efficient and accurate method to evaluate the lattice Green's function.

\section{Efficient computation of the $\mathsf{C}$ matrix}\label{Cmatrix}

 Because of its definition by Eq.~(\ref{DefC}), the $\mathsf{C}$ matrix shares many properties with the resolvent $\mathsf{R}(z)=(\mathsf{I}-z \mathsf{A}_G)^{-1}$ also known as the lattice Green's function. In particular, this  implies the existence of recurrence relations between entries of $\mathsf{C}$, which we use to calculate $\mathsf{C}$ in a computationally efficient and numerically stable way. 
 
    \subsection{Recurrence relations}
Let $O:=(0,0)$ denote the origin of the square lattice. For every vertex $v$, let the pair $(i_v,j_v)\in\mathbb{Z}^2$ (or $(i,j)$ if there is no possible confusion) denote its coordinates and $\widetilde{v}$ the vertex of coordinates $(j_v,i_v)$ that is, the symmetric of $v$ with respect to the first bisector. 
First, observe that $\mathsf{R}(z)$ satisfies
$\mathsf{R}(z).(z \mathsf{A}_G) = \mathsf{R}(z)-\mathsf{I}
$ and so, by Eq.~\ref{DefC}, 
\begin{align}
\mathsf{C}.\left(\frac{1}{\Lambda}\mathsf{A}_G\right)&=\lim_{z\to1/\Lambda^-}(\mathsf{I}-\mathsf{P}_\Lambda)\mathsf{R}(z).(z\mathsf{A}_G),\label{CRA}\\
&=\lim_{z\to1/\Lambda^-}(\mathsf{I}-\mathsf{P}_\Lambda)\mathsf{R}(z)-(\mathsf{I}-\mathsf{P}_\Lambda),\nonumber\\
&=\mathsf{C} - (\mathsf{I}-\mathsf{P}_\Lambda).\nonumber
\end{align}
Since $G$ is regular, $(\mathsf{P}_\Lambda)_{v,v'}=1/N\to 0$ as $N\to\infty$, then the above yields the following recurrence relation between entries of $\mathsf{C}$,
\begin{equation}\label{RecRelC}
\frac{1}{\Lambda} \sum_{v''\in\mathcal{N}(v')}\mathsf{C}_{v,v''} = \mathsf{C}_{v,v'} - \delta_{v,v'},
\end{equation}
here $\mathcal{N}(v')$ denotes the set of neighbors of vertex $v'$ on $G$, that is the set of vertices at distance 1 of $v'$. 
Multiplying by $\mathsf{A}$ on the left of $\mathsf{C}$ in Eq.~(\ref{CRA}) we obtain the symmetric relation
$$
\frac{1}{\Lambda} \sum_{v''\in\mathcal{N}(v')}\mathsf{C}_{v',v} = \mathsf{C}_{v,v'} - \delta_{v,v'},
$$
confirming that $\mathsf{C}=\mathsf{C}^T$, as expected from  Eq.~(\ref{Cint}). As the lattice $G$ is vertex transitive, we may choose $v$ to be fixed at the origin $O$ and have $v'$ move throughout the lattice to determine $\mathsf{C}$ entirely. Furthermore, $\mathsf{C}_{v,v'}=\mathsf{C}_{v',v}$, $\mathsf{C}_{O,v}=\mathsf{C}_{O,\widetilde{v}}$, $\mathsf{C}_{O,O}=0$ and  all $\mathsf{C}_{O,(i,i)}$ are well known (see below). Therefore, Eq.~(\ref{RecRelC}) is sufficient to determine all entries of $\mathsf{C}$ by solving linear equations and without computing any integral \cite{Cserti2000, Morita1975}. 
In the sequel, the coefficient $\mathsf{C}_{O,v}$ with $v:=(i,j)$ will be denoted by $\coeff{i,j}$, so that we have:
\begin{prop}\label{RecSquare}
On the square lattice we have $\coeff{0,0}=0$, $\coeff{0,1}=-1$,
\begin{subequations}
\begin{align}
\coeff{i,i}&=-\frac4\pi\sum_{k=1}^i \frac1{2k-1}=-\frac4\pi\sum_{k=0}^{i-1} \frac1{2k+1}&\text{for $i\geq 1$}\label{ciiEq}\\
\coeff{0,j}&=4 \coeff{0,j-1}-\coeff{0,j-2}-2\coeff{1,j-1}&\text{for $j\geq 2$} \label{cdt1}\\
\coeff{i,j}&=4\coeff{i,j-1}-\coeff{i,j-2}-\coeff{i-1,j-1}-\coeff{i+1,j-1}&\text{for $j>i\geq 1$} \label{cdt2}
\end{align}
\end{subequations}
\end{prop}

As shown by Eq.~(\ref{ciiEq}), the $\coeff{i,i}$ values are known exactly and are easily computable. Due to the symmetry $\coeff{i,j}=\coeff{j,i}$, we only need establish that it is possible to compute $\coeff{i,j}$ above the first bissector of the plan to obtain all entries of $\mathsf{C}$.
\begin{lemma}\label{compute_the_cij}
    For $0\le i<j$, it is possible to compute $\coeff{i,j}$.
\end{lemma}
\begin{proof}
    We proceed by induction on $j\ge 1$.
    \begin{itemize}
        \item If $j=1$ then $i=0$ and $\coeff{1,0}=\coeff{0,1}=-1$.
        \item Let $j\ge 1$ and assume that for $0\le i<j$, it is possible to compute $\coeff{i,j}$. We show that for $0\le i<j+1$, it is possible to compute $\coeff{i,j+1}$. There are $3$ possibilities.
        \begin{enumerate}
            \item Either $i=0$ then $\coeff{0,j+1}$ can be computed using Equation~\eqref{cdt1} since $j+1\ge 2$.
            \item Or $1\le i \le j-1$ then $\coeff{i,j+1}$ can be computed using Equation~\eqref{cdt2}.
            \item Or $i=j$ then we use Equation~\eqref{cdt2} and the symmetric relation to obtain $2\coeff{j,j+1}=4\coeff{j,j}-\coeff{j,j-1}-\coeff{j-1,j}$.
        \end{enumerate}
    \end{itemize}
    \end{proof}
\vspace{-5mm}
\noindent Graphically, this iteration implements the following traversal of the square lattice: 
   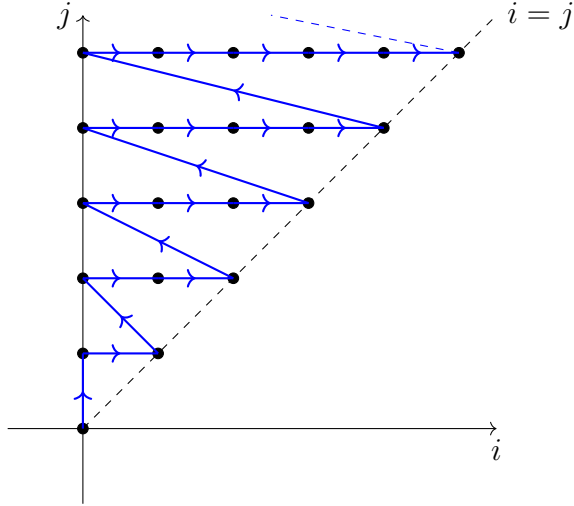
\begin{figure}[htpb]
\begin{center}
\begin{tikzpicture} 
    \draw[->] (-1,0) -- (5.5,0) node[below] {$i$};
    \draw[->] (0,-1) -- (0,5.5) node[left] {$j$};
    \draw[dashed] (0,0) -- (5.5,5.5) node[right] {$i = j$};
    \filldraw (0,0) circle (2pt);
    \foreach \y in {1,...,5}
        \foreach \x in {0,...,\y}
            \filldraw (\x,\y) circle (2pt);
    \draw[dashed, color=blue] (5,5) -- (2.5,5.5);
    \begin{scope}[thick,decoration={
        markings,
        mark=at position 0.5 with {\arrow{>}}}
        ] 
    \draw[postaction={decorate}, color=blue] (0,0) -- (0,1);
    \draw[postaction={decorate}, color=blue] (0,1) -- (1,1);
    \draw[postaction={decorate}, color=blue] (1,1) -- (0,2);
    \draw[postaction={decorate}, color=blue] (0,2) -- (1,2);
    \draw[postaction={decorate}, color=blue] (1,2) -- (2,2);
    \draw[postaction={decorate}, color=blue] (2,2) -- (0,3);
    \draw[postaction={decorate}, color=blue] (0,3) -- (1,3);
    \draw[postaction={decorate}, color=blue] (1,3) -- (2,3);
    \draw[postaction={decorate}, color=blue] (2,3) -- (3,3);
    \draw[postaction={decorate}, color=blue] (3,3) -- (0,4);
    \draw[postaction={decorate}, color=blue] (0,4) -- (1,4);
    \draw[postaction={decorate}, color=blue] (1,4) -- (2,4);
    \draw[postaction={decorate}, color=blue] (2,4) -- (3,4);
    \draw[postaction={decorate}, color=blue] (3,4) -- (4,4);
    \draw[postaction={decorate}, color=blue] (4,4) -- (0,5);
    \draw[postaction={decorate}, color=blue] (0,5) -- (1,5);
    \draw[postaction={decorate}, color=blue] (1,5) -- (2,5);
    \draw[postaction={decorate}, color=blue] (2,5) -- (3,5);
    \draw[postaction={decorate}, color=blue] (3,5) -- (4,5);
    \draw[postaction={decorate}, color=blue] (4,5) -- (5,5);
    \end{scope}
\end{tikzpicture}
\caption{Direction of computations for the calculation of all $c_{i,j}$ on the square lattice.}
\label{Direction_of_computations}
\end{center}
\label{parcours}
\vspace{-7mm}
\end{figure} 

\noindent While $\mathsf{C}$ is an infinite matrix, for any given self-avoiding polygon $p$ we need only compute its finite submatrix $\mathsf{C}_p$ to calculate $F_p$ via Eq.~(\ref{Fpres}). Instead of doing this calculation anew for every $p$, in practice, we compute a finite submatrix  of $\mathsf{C}$ only once over the finite grid that contains all the self-avoiding polygons to be constructed (see \S\ref{SS:board} for details). The required $\mathsf{C}_p$ matrices are then  extracted from this.

\subsection{Implementation}
Firstly remark that two vertices belonging to the distance-one neighborhood $\mathcal{N}(p)$ of a polygon $p$ of length $\ell$, stand at a distance of at most $\ell/2+2$ from each other. Thus, computations of the $F_p$ fractions for all  self-avoiding polygons of length up to $\ell$ only necessitates coefficients $c_{i,j}$ for $0\leq i < j \leq \ell/2 + 2$. 
Secondly, by Proposition~\ref{RecSquare}, the computation of any $c_{i,j}$ involves a recursion with many intermediate steps and, consequently, is prone to an accumulation of rounding errors. This leads to a poor overall approximation of $c_{i,j}$, even when working with \emph{double precision floating point}, i.e. with a precision of $10^{-16}$. 
To bypass this issue, observe that since $c_{i,j}\in \mathbb{Q}[1/\pi]$, there is a unique way to write $c_{i,j}=a_{i,j}+b_{i,j}\times \frac{1}\pi$, where $a_{i,j}$ and $b_{i,j}$ are two rationals numbers. Computing these rationals exactly is possible so no accumulation of rounding errors can occur. Adapting Proposition~\ref{RecSquare}, the rational coefficients are obtained through the following recursion:
\begin{enumerate}
\setlength{\itemsep}{1pt}
\item Set $a_{0,0} =0$, $b_{0,0}=0$, $a_{0,1} = -1$ and $b_{0,0} = 0$;
 \item For $i$ from $1$ to $n$ set $a_{i,i} = 0$ and compute $b_{i,i}= b_{i-1,i-1} - 4/(2i-1)$;
 \begin{enumerate}[itemsep=2pt, topsep = -2pt]
     \item Compute $a_{0,j} = 4\,a_{0,j-1} - a_{0,j-2} - 2\,a_{1,j-1}$ and $b_{0,j} = 4\,b_{0,j-1} - b_{0,j-2} - 2\,b_{1,j-1}$;
     \item For $i$ from $1$ to $j-2$, compute $a_{i,j}=4\,a_{i,j-1}-a_{i,j-2}-a_{i-1,j-1}-a_{i+1,j-1}$ and $b_{i,j}=4\,b_{i,j-1}-b_{i,j-2}-b_{i-1,j-1}-b_{i+1,j-1}$;
     \item Compute $a_{j-2,j} = 2\, a_{j-1,j-1}-a_{j-1,j-1}$ and $b_{j-2,j} = 2\, b_{j-1,j-1}-b_{j-1,j-1}$.
 \end{enumerate}
\end{enumerate}
Once coefficients $a_{i,j}$ and $b_{i,j}$ are so obtained, we get a numerical approximation of $c_{i,j}$:
\[
c_{i,j}\approx \left(a_{i,j}+\dfrac{1725033}{5419351} \times b_{i,j}\right) + b_{i,j} \times 2.27595720048157 \times 10^{-15},
\]
which follows from $\dfrac{1}\pi\approx \dfrac{1725033}{5419351} + 2.27595720048157 \times 10^{-15}$.  

Combining these observations and results, and given a maximum length $\ell$, we proceed as above to compute $a_{i,j}$ and $b_{i,j}$ with $0\leq i < j \leq \ell/2 + 2$ once and for all. In a second stage specific to each polygon $p$, the $\mathsf{C}_p$ matrices are constructed from the stored rational coefficients. Rational numbers are given by irreducible fractions $p/q$ of integers, each of which is coded using $64$ bits and so must belong to $\{-2^{63},\ldots,2^{63}-1\}$. If integers outside of that range are needed, an overflow occurs resulting in a missed computation of $c_{i,j}$.
In practice, we observed that integers coded on $64$ bits are enough when computing coefficients~$a_{i,j}$ and $b_{i,j}$ with $0\leq i < j\leq 12$. This means that we can only deal with self-avoiding polygons of length at most $20$.
For longer self-avoiding polygons, we must instead use integers coded on $128$ bits, allowing for valid computations of $c_{i,j}$ with $0\leq i < j \leq 510$, thanks to which we can reach self-avoiding polygons of length at most~$1016$.

\section{An algorithm for constructing self-avoiding polygons}\label{AlgoSec}
Once $\mathsf{C}$ is obtained, the second difficulty of Theorem~\ref{InfiniteSieve} is that $F_p$ depends on the adjacency matrix $\mathsf{B}_p$ of the distance-one neighborhood of the self-avoiding polygon $p$. In order to compute large number of $F_p$ values, we must therefore construct self-avoiding polygons as efficiently as possible.

Let us fix a point $O$ of the square lattice and a length $\ell\geq 2$.
The set of all finite path on the square lattice starting at $O$ will be denoted $\mathcal{P}$.
To each element of $\mathcal{P}$ corresponds a finite sequence of steps  \texttt{Down}, \texttt{Left}, \texttt{Right} and \texttt{Up}.
More precisely there is a bijection between elements of $\mathcal{P}$ and the set $S^\ast$ of finite words on the alphabet $S=\{\Down,\Left,\Right,\Up\}$, where \Down, \Left, \Right and \Up denote steps \texttt{Down}, \texttt{Left}, \texttt{Right} and \texttt{Up} respectively.
From now on, paths on the square lattice starting at $O$ will be identified with finite words on the alphabet $\mathcal{S}$.
The set $\mathcal{P}$ is then equipped with a tree structure $\mathcal{T}$, inherited from that of~$S^\ast$.
The root of $\mathcal{T}$ is the empty word $\varepsilon$ and the sons of a path $w$ are $w\Down$, $w\Left$, $w\Right$ and $w\Up$.
In order to construct each member of the set $\Sap_\ell$ of all self-avoiding polygons of length $\ell$, we will pick out a subtree $\mathcal{T}_\ell$ of $\mathcal{T}$ whose leaves at depth $\ell$ are in bijection with $\Sap_\ell$. 

\subsection{Self-avoiding polygon to word}

A self-avoiding polygon of length $\ell$ can be seen as a path starting from one of its vertices. Yet, this correspondence is not one-to-one. 
For instance, the $1\times 1$ square can be represented by $8$ words on $\mathcal{S}$ depending on the starting vertex and the chosen orientation, as depicted on Figure~\ref{fig:word_square},

\begin{figure}[h!]
\centering
\begin{tikzpicture}
\filldraw[gray](0,0) circle (0.05);
\draw[->](0.05,0) -- (0.95,0);
\draw[->](1,0.05) -- (1,0.95);
\draw[->](0.95,1) -- (0.05,1);
\draw[->](0,0.95) -- (0,0.05);
\draw(0.5,-0.5) node{$\Right\Up\Left\Down$};
\begin{scope}[shift={(0,-2.5)}]
\filldraw[gray](0,0) circle (0.05);
\draw[<-](0.05,0) -- (0.95,0);
\draw[<-](1,0.05) -- (1,0.95);
\draw[<-](0.95,1) -- (0.05,1);
\draw[<-](0,0.95) -- (0,0.05);
\draw(0.5,-0.5) node{$\Up\Right\Down\Left$};
\end{scope}
\end{tikzpicture}
\hspace{2em}
\begin{tikzpicture}
\filldraw[gray](1,0) circle (0.05);
\draw[->](0.05,0) -- (0.95,0);
\draw[->](1,0.05) -- (1,0.95);
\draw[->](0.95,1) -- (0.05,1);
\draw[->](0,0.95) -- (0,0.05);
\draw(0.5,-0.5) node{$\Up\Left\Down\Right$};
\begin{scope}[shift={(0,-2.5)}]
\filldraw[gray](1,0) circle (0.05);
\draw[<-](0.05,0) -- (0.95,0);
\draw[<-](1,0.05) -- (1,0.95);
\draw[<-](0.95,1) -- (0.05,1);
\draw[<-](0,0.95) -- (0,0.05);
\draw(0.5,-0.5) node{$\Left\Up\Right\Down$};
\end{scope}
\end{tikzpicture}
 \hspace{2em}
\begin{tikzpicture}
\filldraw[gray](1,1) circle (0.05);
\draw[->](0.05,0) -- (0.95,0);
\draw[->](1,0.05) -- (1,0.95);
\draw[->](0.95,1) -- (0.05,1);
\draw[->](0,0.95) -- (0,0.05);
\draw(0.5,-0.5) node{$\Left\Down\Right\Up$};
\begin{scope}[shift={(0,-2.5)}]
\filldraw[gray](1,1) circle (0.05);
\draw[<-](0.05,0) -- (0.95,0);
\draw[<-](1,0.05) -- (1,0.95);
\draw[<-](0.95,1) -- (0.05,1);
\draw[<-](0,0.95) -- (0,0.05);
\draw(0.5,-0.5) node{$\Down\Left\Up\Right$};
\end{scope}
\end{tikzpicture}
 \hspace{2em}
\begin{tikzpicture}
\filldraw[gray](0,1) circle (0.05);
\draw[->](0.05,0) -- (0.95,0);
\draw[->](1,0.05) -- (1,0.95);
\draw[->](0.95,1) -- (0.05,1);
\draw[->](0,0.95) -- (0,0.05);
\draw(0.5,-0.5) node{$\Down\Right\Up\Left$};
\begin{scope}[shift={(0,-2.5)}]
\filldraw[gray](0,1) circle (0.05);
\draw[<-](0.05,0) -- (0.95,0);
\draw[<-](1,0.05) -- (1,0.95);
\draw[<-](0.95,1) -- (0.05,1);
\draw[<-](0,0.95) -- (0,0.05);
\draw(0.5,-0.5) node{$\Right\Down\Left\Up$};
\end{scope}
\end{tikzpicture}
\caption{The $1\times 1$  square can be represented by eight different words depending on the chosen starting point (in gray).}
\label{fig:word_square}
\end{figure}

\noindent To avoid multiple constructions of the same self-avoiding polygon, we must fix the starting point and the orientation of the path.
Let $p$ be a self-avoiding polygon. 
The \emph{base line} of $p$ is the horizontal line containing the furthest south edge of $p$. 
The \emph{base point} of $p$ is then the furthest west vertex of $p$ on its base line, see Figure~\ref{fig:base_point1}.
For the rest of this part, self-avoiding polygons will be drawn on the square lattice in such a way that their base point will be a fixed point $O$ of the square lattice.

\begin{figure}[h!]
\centering
\begin{tikzpicture}[x=16pt,y=16pt]
\foreach \x in {-8,...,-4} {
\draw[dashed,gray] (\x,0) -- (\x,4);
}
\foreach \y in {0,...,4} {
\draw[dashed,gray] (-8,\y) -- (-4,\y);
}
\draw[gray,line width=1pt](-8,1) -- (-4,1);

\draw[line width=2pt] (-6,1)--(-5,1);
\draw[line width=2pt] (-5,2)--(-5,1);
\draw[line width=2pt] (-5,3)--(-5,2);       
\draw[line width=2pt] (-5,3)--(-6,3);
\draw[line width=2pt] (-7,3)--(-6,3);
\draw[line width=2pt] (-7,3)--(-7,2);
\draw[line width=2pt] (-6,2)--(-7,2);
\draw[line width=2pt] (-6,2)--(-6,1);
\draw[->](-5.95,0.7) -- (-5.05,0.7);
\filldraw[gray] (-6,1) circle (0.15);
\draw(-6.4,0.6) node{\small$O$};
\end{tikzpicture}
\caption{Base line and base point $O$, of self-avoiding polygon. The arrow denotes chosen path orientation.}
\label{fig:base_point1}
\end{figure}
\noindent At this point a self-avoiding polygon can be represented by two paths of $\mathcal{P}$, depending of the chosen orientation. By construction of the base point these words are of the type $\Right\ldots\Down$ and $\Up\ldots\Left$. 
Hence, forcing the first letter to be a right $\Right$, we fix the polygon orientation. This leads to:
\begin{prop}
Each self-avoiding polygon $p$ of length $\ell\geq 2$ with base point at $O$, is described by a unique word $w(p) = \Right\ldots \Down$ of length $\ell$ on the letters $\mathcal{S}$. 
\end{prop}

\noindent For example, the word corresponding to the polygon shown in Figure~\ref{fig:base_point1} is $\Right\Up\Up\Left\Left\Down\Right\Down$.

\subsection{The tree $\mathcal{T}_\ell$ of $\ell$-admissible words}

We obtain the tree $\mathcal{T}_\ell$ from $\mathcal{T}$ by cutting the latter. More precisely, a node of $\mathcal{T}$ is retained in $\mathcal{T}_\ell$ if and only if it is \emph{$\ell$-admissible}.

\begin{defi}
A word $w$ on the alphabet $\mathcal{S}$ is \emph{$\ell$-admissible} if the following properties are satisfied:

-- $(P1)$ $w$ is of length at most $\ell$;

-- $(P2)$ whenever $w$ is not empty, its first letter is $\Right$;

-- $(P3)$ $w$ is the empty word or its father is itself $\ell$-admissible;

-- $(P4)$ $w$ admits $O$ as base point;

-- $(P5)$ whenever $w$ has length $\ell$, it correspond to a closed path

-- $(P6)$ $w$ has no self intersection;

-- $(P7)$ $w$ has a chance to be completed as a self-avoiding closed path of length $\ell$.

\end{defi}

\noindent Properties $(P1)$ and $(P2$) are immediate to check. 
As we will construct the tree $\mathcal{T}_\ell$ by induction from the empty word, property $(P3)$ causes no particular difficulties.
For property $(P4)$, we have to define the base point of a word $w$, corresponding of a path $p$ starting at $O$.
As explained in the preceding section, the base line of a polygon $p$ is the furthest south horizontal line containing an edge of $p$ and the base point of $w$ is then the furthest west vertex of $p$ on its base line.

\begin{figure}[h!]
\centering
\begin{tikzpicture}[x=15pt,y=15pt]
\foreach \x in {-7,...,-4} {
\draw[dashed,gray] (\x,0) -- (\x,3);
}
\foreach \y in {0,...,3} {
\draw[dashed,gray] (-7,\y) -- (-4,\y);
}
\draw[gray,line width=1pt](-7,1) -- (-4,1);
\draw[line width=2pt] (-6,1)--(-6,2);
\draw[line width=2pt] (-6,2)--(-5,2);
\filldraw[gray] (-6,1) circle (0.15);
\draw(-6.4,0.6) node{\small$O$};
\draw(-5.5,-0.5) node{\phantom{b)}};
\draw(-5.5,-0.5) node{\textbf{a)}};
\end{tikzpicture}
\hspace{1em}
\begin{tikzpicture}[x=15pt,y=15pt]
\foreach \x in {-7,...,-4} {
\draw[dashed,gray] (\x,-1) -- (\x,2);
}
\foreach \y in {-1,...,2} {
\draw[dashed,gray] (-7,\y) -- (-4,\y);
}
\draw[gray,line width=1pt](-7,0) -- (-4,0);
\draw[line width=2pt] (-6,1)--(-5,1);
\draw[line width=2pt] (-5,1)--(-5,0);
\filldraw[black] (-6,1) circle (0.15);
\filldraw[gray] (-5,0) circle (0.15);
\draw(-6.4,0.6) node{\small$O$};
\draw(-5.5,-1.5) node{\textbf{b)}};
\end{tikzpicture}
\hspace{1em}
\begin{tikzpicture}[x=15pt,y=15pt]
\foreach \x in {-8,...,-4} {
\draw[dashed,gray] (\x,0) -- (\x,3);
}
\foreach \y in {0,...,3} {
\draw[dashed,gray] (-8,\y) -- (-4,\y);
}
\draw[gray,line width=1pt](-8,1) -- (-4,1);
\draw[line width=2pt] (-6,1)--(-5,1);
\draw[line width=2pt] (-5,2)--(-5,1);
\draw[line width=2pt] (-5,2)--(-7,2);
\draw[line width=2pt] (-7,2)--(-7,1);
\filldraw[black] (-6,1) circle (0.15);
\filldraw[gray] (-7,1) circle (0.15);
\draw(-6.4,0.6) node{\small$O$};
\draw(-6,-0.5) node{\phantom{b)}};
\draw(-6,-0.5) node{\textbf{c)}};
\end{tikzpicture}
\caption{Three words that are not admissible. In \textbf{a)}, property $(P2)$ is unsatisfied as the first letter is $\Up$. In \textbf{b)} and \textbf{c)} this is property $(P4)$. Base points of drawn paths are in grey.}
\label{fig:base_point}
\end{figure}

\noindent Since we aim at constructing hundred of billions of self-avoiding polygons, we need a process to check properties $(P4)$ to $(P7)$ as efficiently as possible. 

\subsection{The game board}
\label{SS:board}

We construct our $\ell$-admissible paths on a \emph{game board}, which can be seen as a finite part of the dual lattice of the square lattice with  extra decorations on each cell. The cell containing the base point $O$ is the termed the \emph{base cell} and denoted $\texttt{O}$. We first determine the optimal size of our game board. 
Denote by $E_\ell, N_\ell$, and $W_\ell$ the $\ell$-accepting words that goes furthest east, north and west from the base point $O$, respectively. We verify immediately that these words are of the form 
\begin{align*}
E_\ell & = \underbrace{\Right\cdots\Right}_{a_\ell}\Up\underbrace{\Left\ldots\Left}_{a_\ell}\Down & \text{with $a_\ell=\ell/2-1$,}\\
N_\ell & = \Right\underbrace{\Up\cdots\Up}_{a_\ell}\Left\underbrace{\Down\ldots\Down}_{a_\ell}\\
W_\ell & = \Right\Up\Up\Left\underbrace{\Left\ldots\Left}_{b_\ell}\Down\underbrace{\Right\ldots\Right}_{b_\ell}\Down&\text{with $b_\ell=\ell/2-3$},
\end{align*}
where we consider only even lengths $\ell$, since there are no polygon of odd length on the infinite square lattice.
It follows that the game board must have at least $a_\ell$ cells east of the base cell,  $a_\ell$ cells  north of it and $b_\ell-1$ cells west of the cell that is immediately north of the base cell.
It will be useful to make a \emph{border} of forbidden cells to our game board. By definition of base point, any cells west of it are forbidden.
Thus the game board has width $w_\ell = a_\ell + b_\ell + 3 = \ell - 1$ and height $h_\ell = a_\ell + 3 = \ell/2 +2$, see the illustration of  Figure~\ref{fig:game_board}.
\begin{figure}[h!]
\centering
\begin{tikzpicture}[x=15pt,y=15pt]
\filldraw[color=gray] (-8.5,-0.5) rectangle (-1.5,0.5);
\filldraw[color=gray] (-8.5,-0.5) rectangle (-7.5,5.5);
\filldraw[color=gray] (-7.5,5.5) rectangle (-1.5,4.5);
\filldraw[color=gray] (-2.5,4.5) rectangle (-1.5,0.5);
\filldraw[color=gray] (-7.5,0.5) rectangle (-6.5,1.5);
\foreach \x in {-8,...,-1} {
\draw(\x-0.5,-0.5) -- (\x-0.5,5.5);
}
\foreach \y in {0,...,6} {
\draw (-8.5,\y-0.5) -- (-1.5,\y-0.5);
}
\draw[line width=2pt] (-6,1)--(-3,1);
\draw[line width=2pt] (-3,1)--(-3,2);
\draw[line width=2pt] (-3,2)--(-6,2);       
\draw[line width=2pt] (-6,2)--(-6,1);
\filldraw[gray] (-6,1) circle (0.15);
\draw(-6,-1) node{\phantom{b)}};
\draw(-5,-1) node{\textbf{a)}};
\end{tikzpicture}
\hspace{1em}
\begin{tikzpicture}[x=15pt,y=15pt]
\filldraw[color=gray] (-8.5,-0.5) rectangle (-1.5,0.5);
\filldraw[color=gray] (-8.5,-0.5) rectangle (-7.5,5.5);
\filldraw[color=gray] (-7.5,5.5) rectangle (-1.5,4.5);
\filldraw[color=gray] (-2.5,4.5) rectangle (-1.5,0.5);
\filldraw[color=gray] (-7.5,0.5) rectangle (-6.5,1.5);
\foreach \x in {-8,...,-1} {
\draw(\x-0.5,-0.5) -- (\x-0.5,5.5);
}
\foreach \y in {0,...,6} {
\draw (-8.5,\y-0.5) -- (-1.5,\y-0.5);
}
\draw[line width=2pt] (-6,1)--(-5,1);
\draw[line width=2pt] (-5,1)--(-5,4);
\draw[line width=2pt] (-5,4)--(-6,4);       
\draw[line width=2pt] (-6,4)--(-6,1);
\filldraw[gray] (-6,1) circle (0.15);
\draw(-5,-1) node{\textbf{b)}};
\end{tikzpicture}
\hspace{1em}
\begin{tikzpicture}[x=15pt,y=15pt]
\filldraw[color=gray] (-8.5,-0.5) rectangle (-1.5,0.5);
\filldraw[color=gray] (-8.5,-0.5) rectangle (-7.5,5.5);
\filldraw[color=gray] (-7.5,5.5) rectangle (-1.5,4.5);
\filldraw[color=gray] (-2.5,4.5) rectangle (-1.5,0.5);
\filldraw[color=gray] (-7.5,0.5) rectangle (-6.5,1.5);
\foreach \x in {-8,...,-1} {
\draw(\x-0.5,-0.5) -- (\x-0.5,5.5);
}
\foreach \y in {0,...,6} {
\draw (-8.5,\y-0.5) -- (-1.5,\y-0.5);
}
\draw[line width=2pt] (-6,1)--(-5,1);
\draw[line width=2pt] (-5,2)--(-5,1);
\draw[line width=2pt] (-5,3)--(-5,2);       
\draw[line width=2pt] (-5,3)--(-6,3);
\draw[line width=2pt] (-7,3)--(-6,3);
\draw[line width=2pt] (-7,3)--(-7,2);
\draw[line width=2pt] (-6,2)--(-7,2);
\draw[line width=2pt] (-6,2)--(-6,1);
\filldraw[gray] (-6,1) circle (0.15);
\draw(-5,-1) node{\textbf{c)}};
\end{tikzpicture}
\hspace{1em}
\caption{Game board for $\ell=8$. The base cell is depicted with a gray dot. Forbidden cells are in gray. Paths $E_\ell$, $N_\ell$ and $W_\ell$ are drawn on \textbf{a)}, \textbf{b)} and \textbf{c)}
 respectively.}
 \label{fig:game_board}
\end{figure}
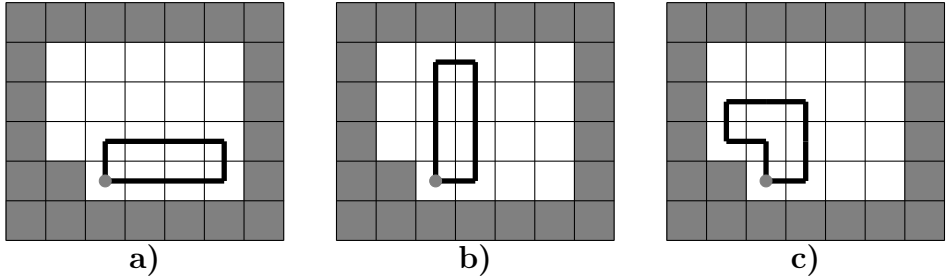

\noindent We identify each cell \texttt{c} of the game board by a unique integer $n_\texttt{c}$ in $\{0,\ldots,h_\ell\times w_\ell-1\}$: we have $n_\texttt{c} = w_\ell\times y + x$ for a cell \texttt{c} at position $(x,y)$ with the convention that the left-bottom cell has position $(0,0)$ and the right-top one has position $(w_\ell-1, h_\ell-1)$.  Consider a non forbidden cell \texttt{c} and denote by $\texttt{c}_D, \texttt{c}_L, \texttt{c}_R$ and $\texttt{c}_U$ the cells visited after a \Down, \Left, \Right and \Up steps from \texttt{c}, respectively. Then,
\[
n_{\texttt{c}_D} = n_\texttt{c} - w_\ell, \quad n_{\texttt{c}_L} = n_\texttt{c} -1, \quad n_{\texttt{c}_R} = n_\texttt{c}+ 1,\quad \text{and} \quad n_{\texttt{c}_U} = n_\texttt{c} + w_\ell.
\]

\subsection{Cell data}
In order to check that a path satisfies property $(P7)$, we add information to the game-board cells as follows.  
Let $w$ be a word of length $t$ on $\mathcal{S}$. 
Starting from the base cell \texttt{O}, we construct a sequence of cells~$\phi_w$ of length~$\ell(w)$ accordingly to $w$ : 
for $i\in\{1,\ldots,\ell(w)\}$ the cell $\kappa_w(i)$ is the adjacent cell of $\kappa_w(i-1)$ with respect to $w_i$.
Let $0\leq t<\ell$ be an integer.  In order to be able to complete $w$ into a closed path of length $\ell$, we must complete it with $\ell-t$ steps linking the cell $\kappa_w(\ell(w))$ to the base cell~\texttt{O}.
This is  only possible if the length of the minimal sequence of steps linking $\kappa_w(\ell(w))$ to \texttt{O} is smaller than~$\ell-t$.
For each non forbidden cell \texttt{c} of the game board, we thus define $d(\texttt{c})$ to be the minimal number of steps needed to reach the base cell from \texttt{c}. This quantity is the Manhattan distance between~\texttt{c} and the base cell~\texttt{O}: if we denote by $(0,0)$ and $(x,y)$ positions of the base cell \texttt{O} and \texttt{c}, respectively, we have $d(\texttt{c}) = |x| + |y|$. 
By convention, we set $d(\texttt{c})=+\infty$ for each forbidden cell (in gray in Figure~\ref{fig:game_board}). 
We can now give a precise formulation of $(P7)$ :

\begin{defi}
A word $w$ of length $t$ on $\mathcal{S}$ satisfies $(P7)$ whenever $t + d(c_w(t)) \leq \ell$ holds.
\end{defi}

\noindent Thus, to each cell \texttt{c} of the game board 
we attach two pieces of information:
\begin{itemize}
    \item the Manhattan distance $d(\texttt{c})$ between \texttt{c} and the base cell \texttt{O};
    \item an integer $t(\texttt{c})$, $1\leq t(\texttt{c})\leq \ell$, that depends on the current word $w$ being constructed.
    For~$i\in\{1,\ldots,\ell(w)\}$, the value of $t(\kappa_w(i))$ is set to $i$. Roughly speaking, $t(\texttt{c})$ specifies \emph{when} the cell $\texttt{c}$ is reached by the word $w$.
\end{itemize}

\subsection{Initialization}
We now describe the main algorithmic engine for the construction of the self-avoiding polygons of a fixed length~$\ell$. 
Recall that this is based on the exploration of the tree $\mathcal{T}_\ell$, whose leaves are in bijection with $\Sap_\ell$.
We begin by constructing the game board \texttt{G} as described above in \S~\ref{SS:board}. We then compute the Manhattan distance $d(\texttt{c})$ from the base cell $\texttt{O}$ to each cell \texttt{c} of \texttt{G}.
In addition of the game board \texttt{G}, we use 
\begin{enumerate}
\item an array \texttt{word} of size $\ell$ storing the word $w$ under construction, \texttt{word[$i$]} being the $i$th letter of $w$;
\item an array \texttt{kappa} of size $\ell$ containing integers in $\{0, \ldots, h_\ell\times w_\ell - 1\}$ or $+\infty$ (in practice we can replace $+\infty$ by any integer greater than $h_\ell\times w_\ell$);
\item a stack \texttt{stack} of triple $(\texttt{c}, k, s)$, where \texttt{c} is a cell, $k$ an integer between $0$ and $\ell$, and $s\in \mathcal{S}$ is a step.
\end{enumerate}

\noindent We give a complete description of the initialization steps in Algorithm~\ref{A:Init}.
\begin{algorithm}[h]
\small
\caption{\small Initialization for the construction of $\Sap_\ell$.}\label{A:Init}
\begin{algorithmic}[1]
\Procedure{InitEnumSap}{}
	\State Construct the game board $\texttt{G}$
    \For{each cell \texttt{c} of $\texttt{G}$}
    \If{\texttt{c} is on the border}
        \State $d(\texttt{c}) \gets +\infty$ 
    \Else
        \State $d(\texttt{c}) \gets$ Manhattan distance between \texttt{c} and \texttt{O}
        \State $t(\texttt{c}) \gets -1$
    \EndIf
    \EndFor
    \State \texttt{word} $\gets$ array steps of length $\ell$
    \State \texttt{kappa} $\gets$ array of cells of length $\ell$
    \State \texttt{stack} $\gets$ empty stack
    \State \texttt{stack.push(\texttt{O}, $0$, $\Right$)} \Comment{Start with a step \Right from the base cell \texttt{O}, to ensure (P2) and (P4)}
\EndProcedure
\end{algorithmic}
\end{algorithm}

\subsection{Exploration of the tree~$\mathcal{T}_\ell$}

We perform a  \emph{depth-first} exploration of tree of $\mathcal{T}_\ell$ using the stack \texttt{stack} which is a container equipped with four operations: 
\begin{itemize}
    \item \texttt{stack.empty()} testing if the stack is empty;
    \item \texttt{stack.push(...)} putting an element at the top of the stack;
    \item \texttt{stack.top()} returning the element at the top of the stack;
    \item \texttt{stack.pop()} removing the element at the top of the stack.
\end{itemize}
Consider for instance the set $\Sap_8$ of length 8 self-avoiding polygons and suppose that at some point of the exploration of $\mathcal{T}_8$ we pop the $8$-admissible word $w = \Right\Up\Right\Up$ from the stack \texttt{stack}.
As a word, $w$ can be completed as $w_1 = w \Down$, $w_2=w \Left$, $w_3 = w\Right$ or $w_4 = w \Up$. 

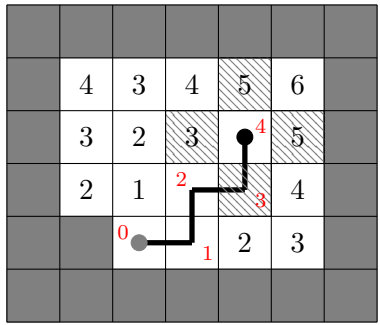
\begin{figure}[H]
\begin{center}
\begin{tikzpicture}[x=20pt,y=20pt]
\filldraw[color=gray] (-8.5,-0.5) rectangle (-1.5,0.5);
\filldraw[color=gray] (-8.5,-0.5) rectangle (-7.5,5.5);
\filldraw[color=gray] (-7.5,5.5) rectangle (-1.5,4.5);
\filldraw[color=gray] (-2.5,4.5) rectangle (-1.5,0.5);
\filldraw[color=gray] (-7.5,0.5) rectangle (-6.5,1.5);
\foreach \x in {-8,...,-1} {
\draw(\x-0.5,-0.5) -- (\x-0.5,5.5);
}
\foreach \y in {0,...,6} {
\draw (-8.5,\y-0.5) -- (-1.5,\y-0.5);
}
\draw[line width=2pt] (-6,1)--(-5,1);
\draw[line width=2pt] (-5,1)--(-5,2);
\draw[line width=2pt] (-5,2)--(-4,2);
\draw[line width=2pt] (-4,2)--(-4,3);
\draw[pattern=north west lines, pattern color=gray]  (-3.5,2.5) rectangle ++ (1,1);
\draw[pattern=north west lines, pattern color=gray]  (-5.5,2.5) rectangle ++ (1,1);
\draw[pattern=north west lines, pattern color=gray]  (-4.5,1.5) rectangle ++ (1,1);
\draw[pattern=north west lines, pattern color=gray]  (-4.5,3.5) rectangle ++ (1,1);

\draw (-7, 2) node {\small $2$};
\draw (-7, 3) node {\small $3$};
\draw (-7, 4) node {\small $4$};
\draw (-6, 2) node {\small $1$};
\draw (-6, 3) node {\small $2$};
\draw (-6, 4) node {\small $3$};
\draw (-5, 3) node {\small $3$};
\draw (-5, 4) node {\small $4$};
\draw (-4, 1) node {\small $2$};
\draw (-4, 4) node {\small $5$};
\draw (-3, 1) node {\small $3$};
\draw (-3, 2) node {\small $4$};
\draw (-3, 3) node {\small $5$};
\draw (-3, 4) node {\small $6$};

\filldraw[gray] (-6,1) circle (0.15);
\filldraw[black] (-4,3) circle (0.15);
\draw[red] (-3.7, 3.2) node {\scriptsize $4$};
\draw[red] (-3.7, 1.8) node {\scriptsize $3$};
\draw[red] (-5.2, 2.2) node {\scriptsize $2$};
\draw[red] (-4.7, 0.8) node {\scriptsize $1$};
\draw[red] (-6.3, 1.2) node {\scriptsize $0$};

\end{tikzpicture}
\end{center}
 \caption{Depth first exploration of the tree $\mathcal{T}_8$. 
 The $8$-admissible word $w = \Right\Up\Right\Up$ of length $4$ is represented by the solid black line on the game board. The base cell is that with the gray bullet, while cell $\kappa_w(4)$ is highlighted with a black bullet. Black numbers in cells $\texttt{c}$ stand for $d(\texttt{c})$, red numbers give $t(\texttt{c})$. Shaded cells are those immediately reachable from $w$ by adding one step. They yield words $w_1=w\Down,\, w_2=w\Left, \, w_3=\Right$ and $\, w_4=w\Up$.}\label{F:ExploreExample}
\end{figure}
\noindent Using the game board as depicted in Figure~\ref{F:ExploreExample}, we check whether these four words are $8$-admissible. Let us denote by $\texttt{c}_1, \texttt{c}_2, \texttt{c}_3$ and $\texttt{c}_4$ the cells reached by $w_1, w_2, w_3$ and $w_4$ respectively, \emph{i.e.}, $\texttt{c}_i = \kappa_{w_i}(5)$. 
In the case of $w_1$, noting that $t(\texttt{c}_1) = 3$ is lower than the length of $w_1$, this word may self-intersect. Remark that $t(\texttt{c}_1) = 3<5$ is not sufficient to immediately conclude that $w_1$ self intersects because, for the sake of speed and efficiency, the game board information is not systematically cleared. As a consequence many cells $\texttt{c'}$ could share the same value  $t(\texttt{c}')=t(\texttt{c}_1)$, the values taken by $t$ being correct only for cells visited by the current word. So, to decide if $w_1$ self intersects the algorithm checks whether cell $\texttt{c}_1$ is indeed the third one visited by $w_1$. This is what $\texttt{kappa}$ is used for: given that $\kappa_{w_1}(3)=\kappa_{w}(3) = \texttt{c}_1$, $\texttt{c}_1$ is really visited twice and $w_1$ is rejected because it violates property $(P6)$.
Words $w_2, w_3$ and $w_4$ are checked similarly and found to satisfy $(P6)$. There remains to verify property $(P7)$ for them. These words are of length $5$ leaving only three steps to reach back to the base cell. 
As the distance information contained in cells $\texttt{c}_3$ and $\texttt{c}_4$ is $5$, which is strictly greater than the three remaining steps, the algorithm rejects words $w_3$ and $w_5$ for violating $(P7)$. Hence $w_2=w\Left$ is the sole $8$-admissible word that can be built from $w$ and $(\texttt{c}_2, 5, \Left)$ is pushed on the stack \texttt{stack}. The exploration of $\mathcal{T}_8$ continues in this fashion until the stack \texttt{stack} is empty.\\

The algorithm for the complete exploration of $\mathcal{T}_\ell$ is presented in Algorithm~\ref{A:Enum}. 
\begin{algorithm}[htpb]
\small
\caption{\small Construction of $\Sap_\ell$.}\label{A:Enum}
\begin{algorithmic}[1]
\Procedure{EnumSap}{}
 \While{\texttt{stack} not empty}
 \State (\texttt{c}, $k$, $s$) $\gets$ \texttt{stack.top()}
 \State \texttt{stack.pop()}
 \State \textsc{Apply}(\texttt{c}, $k$, $s$)
 \State \texttt{c'} $\gets$ the cell adjacent to $\texttt{c}$ with respect to $s$
 \If {$k = \ell$}
 \State \textsc{Treat}(\texttt{word})
 \Else
 \For{$s'$ in $\{\Down, \Left, \Right, \Up \}$}
 \If {\textsc{canAdd}(\texttt{c'}, $k+1$, $s'$)}
 \State \texttt{stack.push}(\texttt{c'}, $k+1$, $s'$)
 \EndIf
 \EndFor
 \EndIf
 \EndWhile
\EndProcedure
\end{algorithmic}
\end{algorithm}
\newpage
\noindent $\triangleright$ In line \texttt{5}, the current word is augmented with step $s$, which corresponds to exploring an edge of $\mathcal{T}_\ell$. The updates performed by the code on that occasion are presented in Algorithm~\ref{A:Apply}.\\
$\triangleright$ In line \texttt{8}, the algorithm has reached a leaf of $\mathcal{T}_\ell$ at depth $\ell$. At this point, \texttt{word} codes for a valid self-avoiding polygon $p$ of length $\ell$. The function \texttt{treat} is called and it implements any desired computations on~$p$, stores it for future use, etc.\\
$\triangleright$ In line \texttt{11}, the algorithm checks whether the current word augmented with  step \texttt{s'} is $\ell$-admissible, a task performed by Algorithm~\ref{A:CanAdd}. If the answer is positive we push this future exploration on the stack.

\begin{algorithm}[h]
\small
\caption{\small Apply a new step to the on construction path.}\label{A:Apply}
\begin{algorithmic}[1]
\Function{Apply}{\texttt{c}, $k$, $s$}
 \State $t(\texttt{c}) \gets$ $k$
 \State \texttt{word[$k$]} $\gets$ $s$
  \State \texttt{kappa[$k$]} $\gets$ \texttt{c}
 \State \texttt{kappa[$k+1$]} $\gets +\infty$
\EndFunction
\end{algorithmic}
\end{algorithm}

\begin{algorithm}[h!]
\small
\caption{\small Test if we can add the letter $s$ at position $k$ in the word under construction}\label{A:CanAdd}
\begin{algorithmic}[1]
\Function{CanAdd}{\texttt{c}, $k$, $s$}
 \State \texttt{c'} $\gets$ the cell adjacent to $\texttt{c}$ with respect to $s$
  \If{$k + d(\texttt{c'}) \geq \ell$} \Comment{Checks $(P7)$ and whether a game board boundary is reached}
  \State \textbf{return} \texttt{false}
  \EndIf
 \If{$t(\texttt{c'}) \leq k $ and $\texttt{kappa[}t(\texttt{c'})\texttt{]} = \texttt{c'}$} \Comment{Checks $(P6)$}
  \State \textbf{return} \texttt{false}
 \EndIf
 \State \textbf{return} \texttt{true}
 \EndFunction
\end{algorithmic}
\end{algorithm}

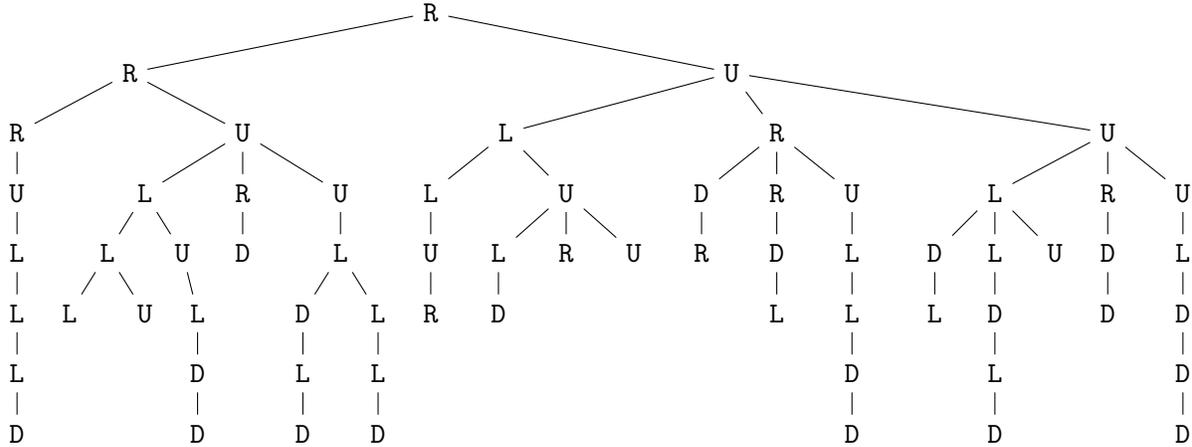
\begin{figure}
\begin{center}
\small
\begin{tikzpicture}[level distance = 8mm,
level 1/.style={sibling distance=80mm},
level 2/.style={sibling distance=50mm},
level 3/.style={sibling distance=10mm}]
  \node {\Right}
    child {node {\Right}
      child [xshift = 10 mm]{node {\Right}
        child {node {\Up}
          child {node {\Left}
            child {node {\Left}
              child {node {\Left}
                child {node {\Down}}
              }
            }
          }
        }
      }
      child[xshift = -10mm]{node {\Up}
        child[xshift = -3mm] {node {\Left}
          child {node {\Left}
            child {node {\Left}}
            child {node {\Up}}
          }
          child {node {\Up}
            child[xshift = 2mm] {node {\Left}
              child {node {\Down}
                child {node {\Down}}
              }
            }
          }
        }
        child {node {\Right}
          child {node {\Down}}
        }
        child[xshift = 3mm] {node {\Up}
          child {node {\Left}
            child {node {\Down} 
              child {node {\Left}
                child {node {\Down}}    
              }
            }
            child {node {\Left}
              child {node {\Left}
                child {node {\Down}}    
              }
            }      
          }
        }
      }
    }
    child {node {\Up}
      child[xshift=20mm]  {node {\Left}
        child[xshift=-5mm] {node {\Left}
           child {node {\Up}
            child {node {\Right}}
          }
        }
        child[xshift=3mm] {node {\Up}
          child[xshift=1mm] {node {\Left}
          child {node {\Down}}
          }
          child {node {\Right}}
          child[xshift=-1mm] {node {\Up}}
        }
      }
      child[xshift=6mm] {node {\Right}
        child {node {\Down}
          child {node {\Right}}
        }
        child {node {\Right}
          child {node {\Down}
            child {node {\Left}}
            }
        }
        child {node {\Up}
          child {node {\Left}
            child {node {\Left}
             child {node {\Down}
               child {node {\Down}}
             }
            }
          }
        }
      }
      child {node {\Up}
        child[xshift=-5mm] {node {\Left}
          child[xshift=2mm] {node {\Down}
           child {node {\Left}}
          }
          child {node {\Left}
            child {node {\Down}
             child {node {\Left}
               child {node {\Down}}
             }
            }
          }
          child[xshift=-2mm] {node {\Up}
          }
        }
        child {node {\Right}
          child {node {\Down}
            child {node {\Down}}
          }
        }
        child {node {\Up}
          child {node {\Left}
            child {node {\Down}
              child {node {\Down}
                child {node {\Down}}
              }
            }
          }
        }
      }
    };
\end{tikzpicture}
\end{center}
\caption{The complete tree $\mathcal{T}_8$. To each node $n$ corresponds the $8$-admissible word made of the letters met going from the root of the tree down to node $n$. The seven elements of $\Sap_8$ are the seven leaves at depth $8$ of $\mathcal{T}_8$.\label{TreeT8}}  
\end{figure}

\section{Computational results}

Our experiments were carried out on the shared computational platform \textsc{Calculco} \cite{Calculco}.
This platform is equipped with 28 nodes for a total of 1852 cores and is designed for distributed computation.
In this section we detail the construction of a suite of programs to take advantage of \textsc{Calculco} to carry out the computation of the sums $S(\ell)$ of \S\ref{Sums} up to $\ell=38$.

\subsection{Memory usage and parallelization}
Our first task is to construct and store the self-avoiding polygons of a given length $\ell$.
The process is fully described in Section~\ref{AlgoSec} so we here focus on the storage method.
As a self-avoiding polygons of length $\ell$ is given by a word which we code  as an array of steps $\{\Down,\Left,\Right,\Up\}$ of length $\ell$.
Each step can be encoded with two bits :
\begin{equation}\label{basicencoding}
\Down \rightarrow \texttt{00}, \quad \Left \rightarrow \texttt{01}, \quad \Right \rightarrow \texttt{10}, \quad \Up \rightarrow \texttt{11}.
\end{equation}

With this methods a self-avoiding polygons of length $\ell$ needs $2 \ell$ bits to be stored. 
Since there are $2\,895\,432\,660$ self-avoiding polygons of length 38, this method requires $5.6\  \texttt{Tb}$ to store all the elements of $\Sap_{38}$. To decrease this enormous quantity of data, remark that thanks to the depth-first exploration of the tree $\mathcal{T}_\ell$, elements of $\Sap_\ell$ are constructed in lexicographic order.
This implies that two consecutive self-avoiding polygons likely share a long prefix. As a consequence, to store two consecutive self-avoiding polygons $p_1$ and $p_2$, we store $p_1$, then store the  length of the common prefix between $p_2$ and $p_1$ as well as the suffix of $p_2$ with respect to this prefix using the previous encoding. Since we consider \Sap of length at most $38\leq 64$, we need no more than $6$ bits to store the length of the prefix, considerably reducing the memory requirement for $p_2$. Finally, we use the Lempel–Ziv–Markov chain algorithm (LZMA) \cite{LZMA} to perform on the fly compression. 
The following table summarize space needed to store $\Sap_\ell$ using basic encoding Eq.~(\ref{basicencoding}); prefix encoding; and LZMA compression on prefix encoding. Empirically, we found that LZMA compression employed directly on basic encoding takes much time and gives terrible results in terms of storage.

\footnotesize{\[
\begin{array}{|c|c|c|c|c|c|c|c|c|c|c|c|}
\hline
\ell & 20 & 22 & 24 & 26 & 28 & 30 & 32 & 34 & 36 & 38 & 40\\
\hline
\text{basic} & 400\texttt{Kb} & 2.4\texttt{Mb} & 4.4\texttt{Mb} & 89\texttt{Mb} & 556\texttt{Mb} & 3.4\texttt{Gb} & 21.6\texttt{Gb} & 137\texttt{Gb} & 877\texttt{Gb}& 5.6 \texttt{Tb}& 35.4 \texttt{Tb} \\
\hline
\text{prefix}&  148\texttt{Kb} & 809\texttt{Kb} & 4.4\texttt{Mb}  & 25.1\texttt{Mb}  & 145\texttt{Mb}& 849\texttt{Mb} & 4.9\texttt{Gb} &- &- &- &- \\
\hline
 \text{+ \text{lzma}}  & 516\texttt{Kb} & 524\texttt{Kb} & 884\texttt{Kb} & 1.8 \texttt{Mb} & 4.4 \texttt{Mb} & 14\texttt{Mb} & 66 \texttt{Mb}& 350\texttt{Mb} & 1.5\texttt{Gb} & 8\texttt{Gb} & 43 \texttt{Gb}\\
\hline
\end{array}
\]}
\normalsize

We can parallelize the exploration of the tree $\mathcal{T}_\ell$ by selecting nodes $w_1, \ldots , w_k$ such that the leaves of maximal depth in $\mathcal{T}_\ell$  are in bijection with leaves of maximal depth in the subtrees of $\mathcal{T}_\ell$ rooted at $w_1, \ldots, w_k$.
The difficulty here is to find a selection of nodes such that the times needed to explore the corresponding rooted subtrees are balanced. 
Using empirical experiments on $\mathcal{T}_{28}$, we were able to parallelize the exploration of $\mathcal{T}_{\ell}$ for $\ell \geq 28$ using $64$ threads. As our ultimate goal is the computation of the fractions $F_p$ for all elements of $\Sap_\ell$, the constructed self-avoiding polygons are stored in many different files, typically several thousands.

\subsection{Computation of $F_p$}\label{FP_pour_p}
At this point, a self-avoiding polygon $p$ is given as a word $w$ of steps. From this word, we draw the polygon back onto the square grid. This yields the graph of the distance one neighborhood $\mathcal{N}(p)$, the adjacency matrix of which is $\mathsf{B}_p$ (see Figure~\ref{SAP}). The matrix $\mathsf{C}_p$ is similarly constructed from the pre-computed coefficients $c_{i,j}$, see \S\ref{Cmatrix}. We then form $\mathsf{M}_p = \mathsf{Id}+\Lambda^{-1} \mathsf{C}_p\cdot \mathsf{B}_p$ and use a Gauss–Jordan elimination to compute the adjugate of $\mathsf{M}_p$.
To speed up computations, we use the FMA instruction set \cite{FMA} to perform fused multiply–add operations on vector of four doubles.
Typically, if \texttt{a}, \texttt{b} and \texttt{a} are arrays of four double floats then the computation of 
\(
\texttt{a}[i] \times \texttt{b}[i] + \texttt{c}[i]
\) for $i=0,1,2,3$,
is done with only one \texttt{CPU} instruction.



\subsection{Numerical results}
All the algorithms presented here were implemented in \texttt{C++}. Source codes are available on \texttt{GitHub} \cite{GitHub}. We present the final numerical results in Table.~\ref{FellTable}, showing the computed values for $F_\ell$ and their accumulation $S(\ell)$. These results are also plotted in Fig.~\ref{FpSum}. These observations lead us to conjecture that $S(\ell) = 1-\ell^{-3/5}+O(\ell^{-1})$ and, from there, $F_\ell=\ell^{-8/5}+O(\ell^{-2})$.
\begin{table}
\begin{center}
\begin{tabular}{|c|c|c|c|c|}
\hline
$\ell$ & $\pi(\ell)$ & $F_\ell$ & $S(\ell)$ \\
\hline
 2 &              1 & 0.50000000000000&0.50000000000000\\
 4 &              1 & 0.14727245910375&0.64727245910375\\
 6 &              2 & 0.06204664274521&0.70931910184896\\
 8 &              7 & 0.04001566383131&0.74933476568027\\
10 &             28 & 0.02805060444094&0.77738537012121\\
12 &            124 & 0.02102490313204&0.79841027325325\\
14 &            588 & 0.01644695527417&0.8148572285274199\\
16 &          2 938 & 0.01329675992709&0.8281539884545099\\
18 &         15 268 & 0.01102242742254&0.8391764158770499\\
20 &         81 826 & 0.00931937541569&0.8484957912927399\\
22 &        449 572 & 0.00800628886867&0.8565020801614099\\
24 &      2 521 270 & 0.00696952442824&0.8634716045896499\\
26 &     14 385 376 & 0.00613458465224&0.8696061892418899\\
28 &     83 290 424 & 0.00545085547657&0.8750570447184599\\
30 &    488 384 528 & 0.00488288725692&0.8799399319753799\\
32 &  2 895 432 660 & 0.00440520148051&0.8843451334558899\\
34 & 17 332 874 364 & 0.00399907262190&0.88834420607779\\
36 &104 653 427 012 & 0.00365046770131&0.8919946737790999\\
38 &636 737 003 384 & 0.00334868856902&0.89534336234812\\
\hline
\end{tabular}
\caption{Number $\pi(\ell)$ of self-avoiding polygons of length $\ell$ constructed by the code (see also \cite[Table C.3]{madras2012self}), corresponding $F_\ell=\sum_{p\in\Sap_\ell}F_p$ values and cumulative sums $S(\ell)$ of the $F_\ell$. \label{FellTable}}
\end{center}
\end{table}

\begin{figure}[t!]
     \centering
\includegraphics[width=.468\textwidth]{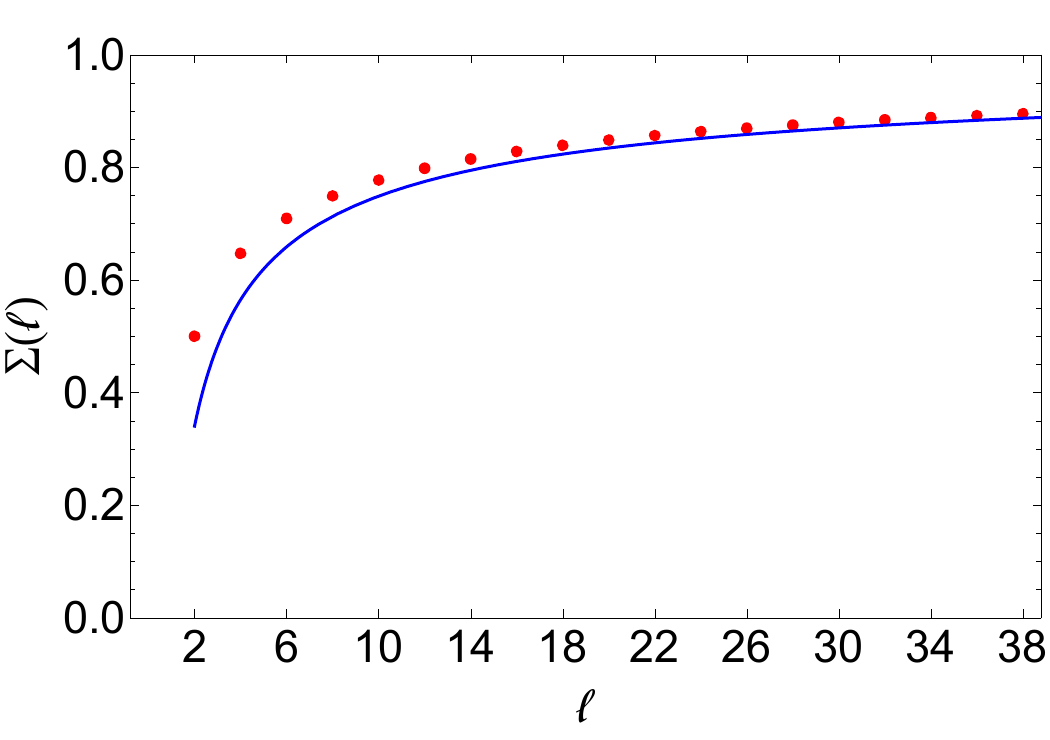}\hspace{6mm}\includegraphics[width=.486\textwidth]{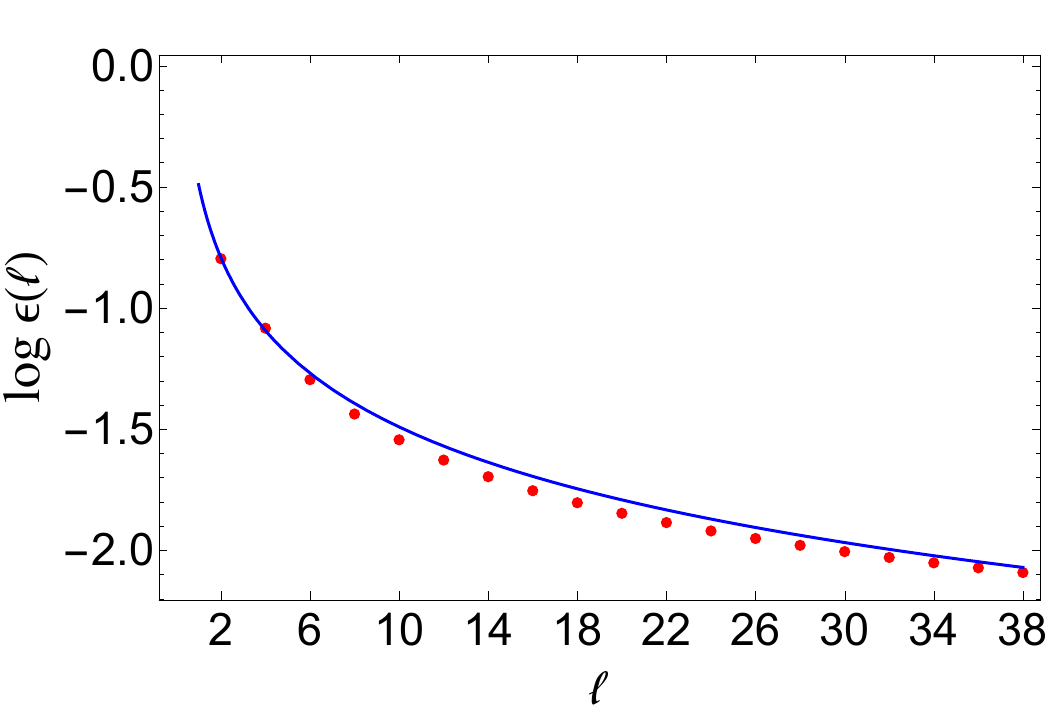}
\vspace{-1mm}
    \caption{Left figure: $S(\ell)$ (red dots), the proportion of closed walks on the infinite square lattice whose last erased loop is a self-avoiding polygon of length at most $\ell$,  as a function of $\ell$. The solid blue line is the conjectured fit by $1-\ell^{-3/5}$, which is supposed to hold asymptotically as $\ell\to\infty$. Right figure: logarithm of the error $\epsilon(\ell) := S(\ell)-(1-\ell^{-3/5})$ (red dots) and fit by the function $0.323/\ell$ (solid blue line).
        Data from Table~\ref{FellTable}.
    }
    \label{FpSum}
\end{figure}

We wrote a second numerical code that can evaluate $F_p$ for one given $p$ of length less than $\ell(p)\leq 1400$. Using this code, we present on Table.~\ref{SquareTable} computed $F_{\text{Sq}_{L\times L}}$ values for the family of $L\times L$ squares on the infinite square lattice as a function of the square's side length $L$. Such results were hitherto completely impossible to access: Theorem~\ref{InfiniteSieve} was unknown until 2021 and only for the six shortest self-avoiding polygons had the fractions $F_p$ been evaluated \cite{Majumdar1991, Manna1992}. The family of squares offers an interesting testing ground for two reasons: we can conjecture the asymptotic growth $F_{\text{Sq}_{L\times L}}$ from both numerical observations and analytical arguments (see below); and empirically we found that $F_{\text{Sq}_{L\times L}}$ is the maximum of $F_p$ values over the set of all self-avoiding polygons of length $4L$. If the length considered is not a multiple of $4$, the self-avoiding polygon with maximum $F_p$ value is a rectangle with a ratio width/height as close to one as possible. At the opposite, the self-avoiding polygons with smallest $F_p$ value are found to be highly winding with 0 interior area.  

The results presented in Table.~\ref{SquareTable} were also calculated \textit{analytically} by a symbolic program in Sage up to and including $\text{Sq}_{7\times 7}$ and by a code for \textsc{Mathematica} up to $\text{Sq}_{30\times 30}$, confirming the numerical results. The analytic formulas are far too cumbersome to be systematically reported here, for example, 
\begin{align*}
F_{\text{Sq}_{3\times 3}} =&\, 26576424-\frac{295147905179352825856}{12301875\, \pi ^{11}}+\frac{191384969764736598016}{2460375\, \pi
   ^{10}}-\frac{156074021897315024896}{1366875\, \pi ^9}\\
   &+\frac{136878648694447013888}{1366875\, \pi
   ^8}-\frac{8851794332131262464}{151875 \pi ^7}+\frac{3588749561696485376}{151875 \,\pi
   ^6}\\&-\frac{38289042343284736}{5625\, \pi ^5}+\frac{870420275786752}{625\, \pi ^4}-\frac{24779053698384}{125\, \pi
   ^3}+\frac{467156948616}{25\, \pi ^2}-\frac{5246537184}{5\, \pi }.
\end{align*}
And the formulas quickly get more involved as the square's side-length increases!

\renewcommand{\arraystretch}{1.05}
\begin{table}
\begin{center}
\begin{tabular}{|c|c||c|c|}
\hline
$L$ & $F_{\text{Sq}_{L\times L}}$&$L$ & $F_{\text{Sq}_{L\times L}}$ \\
\hline
 1 &              $1.8409057387969413\times 10^{-2}$ &
 20 &         $7.234313040400423\times10^{-32}$ \\
 2 &              $4.462339923059934\times 10^{-4}$ &
 30 &        $3.2283875735110397\times10^{-47}$  \\
 3 &              $1.192983879778077\times 10^{-5}$ &
 40 &  $1.4793787629654915\times10^{-62}$ \\
 4 &              $3.2824487567509144\times 10^{-7}$ &
 50 &     $6.877846988396231\times10^{-78}$  \\
5 &             $9.174122974521936\times 10^{-9}$ &
60 &     $3.226927117230214\times10^{-93}$  \\
6 &            $2.5893979305184303\times 10^{-10}$ &
70 &    $1.523552585714086\times10^{-108}$ \\
7&              $7.3577883524995755\times 10^{-12}$&
80 &  $7.226423170276898\times10^{-124}$ \\
8 &            $2.1009188710297932\times 10^{-13}$ &
90 & $3.4396489661899583\times10^{-139}$ \\
9 &          $6.0210656056096115\times 10^{-15}$&
100 &$1.6417501872360198\times10^{-154}$\\
10 &         $1.730587034739647\times 10^{-16}$ &
120 &$3.763918325436204\times10^{-185}$\\
\hline
\end{tabular}
\caption{Numerically computed fraction $F_{\text{Sq}_{L\times L}}$ of closed walks on the infinite square lattice whose last erased loop is the $L \times L$ square $\text{Sq}_{L\times L}$. These numbers were also calculated analytically up to $L=30$. See Fig.~\ref{FpCarre} for a plot of these values.\label{SquareTable}}
\end{center}
\end{table}

\begin{figure}[t!]
     \centering
\includegraphics[width=.6\textwidth]{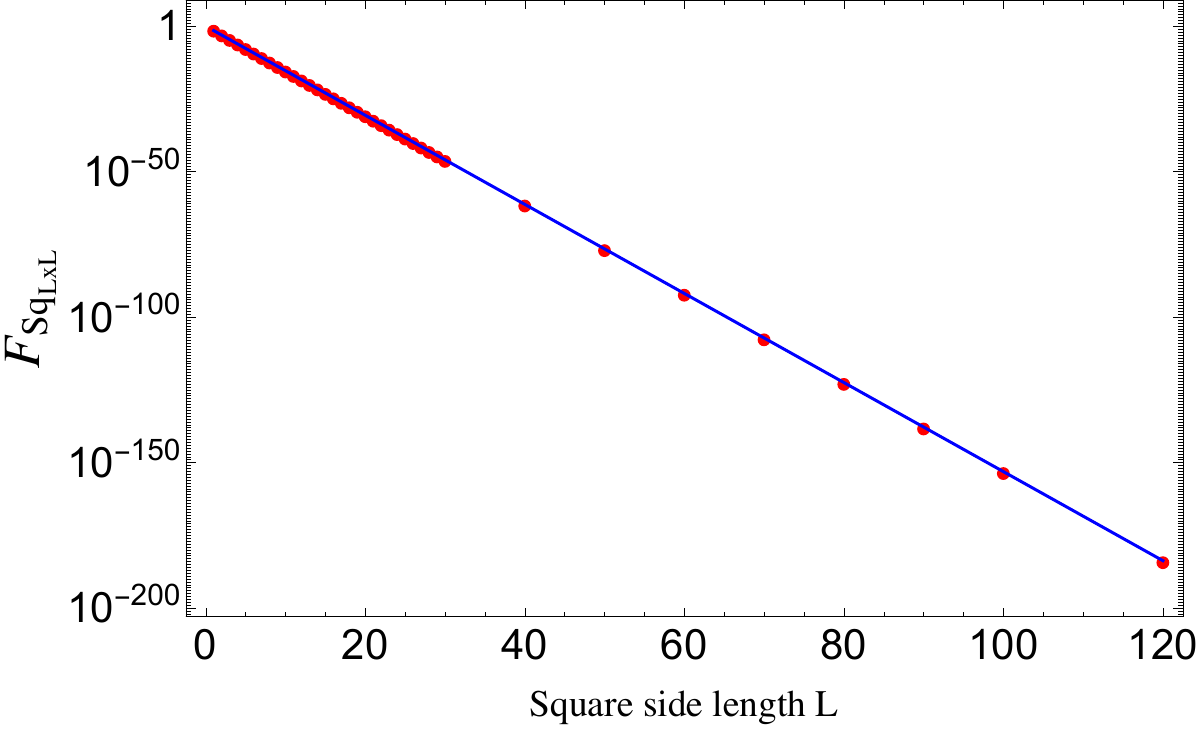}
\vspace{-1mm}
    \caption{Log-plot of the fraction $F_{\text{Sq}_{L\times L}}$ of closed walks on the infinite square lattice whose last erased loop is the $L \times L$ square as a function of the side length $L$ (red dots); and fit based on analytical considerations \S\ref{AnalyticalTriumph} by the function $(\sqrt{2}-1)^{4L}$ (solid blue line). Data from Table~\ref{SquareTable}.
    }
    \label{FpCarre}
\end{figure}

\subsection{Analytical considerations}\label{AnalyticalTriumph}
 Consider the asymptotic behavior of the fraction of closed walks whose last erased loop is the $L\times L$ square as $L\to\infty$. Going back to the roots of the sieve of Theorem~\ref{InfiniteSieve} we have that, on a finite lattice $G_N$ with dominant eigenvalue $\Lambda$, the fraction of all closed walks from a vertex $\bullet$ to itself whose last erased loop is a self-avoiding polygon $p$ visiting $\bullet$ is
$$
F_p=\frac{1}{\Lambda^{\ell(p)}}\frac{\det(\mathsf{Id}-\frac{1}{\Lambda}\mathsf{A}_{G_N\backslash p})}{\det(\mathsf{Id}-\frac{1}{\Lambda}\mathsf{A}_{G_N\backslash\bullet })}.
$$
Here $A_{G_N\backslash\bullet}$ is the adjacency matrix of the finite lattice $G_N$ with vertex $\bullet$ removed. Similarly, $\mathsf{A}_{G_N\backslash p}$ is the adjacency matrix of the finite lattice $G_N$ with all vertices visited by $p$ removed. Since the vertices located in the interior of the self-avoiding polygon $p$ are completely disconnected from those outside of it, the graph $G_N\backslash p$ is disconnected and $F_p$ simplifies to
\begin{equation}
F_p=\frac{1}{\Lambda^{\ell(p)}}\frac{\det(\mathsf{Id}-\frac{1}{\Lambda}\mathsf{A}_{\text{out}(G_N\backslash p)})}{\det(\mathsf{Id}-\frac{1}{\Lambda}\mathsf{A}_{G\backslash\bullet })}\det(\mathsf{Id}-\frac{1}{\Lambda}\mathsf{A}_{\text{in}(G_N\backslash p)}).\label{Fpres1}
\end{equation}
where $\text{in}(G_N\backslash p)$ and $\text{out}(G_N\backslash p)$ designate the set of vertices in the interior and outside of $p$ on $G_N$, respectively.
The induced subgraph of $G_N$ with vertices $\text{in}(G_N\backslash p)$ remains a finite graph under the limit $N\to\infty$ so $\det(\mathsf{Id}-\Lambda^{-1}\mathsf{A}_{\text{in}(G_N\backslash p)})$ is well defined and usually known exactly. For instance on the square lattice and with $p$ the $L\times L$ square, $\text{in}(G_N\backslash \text{Sq}_{L\times L})$ is the $L-2\times L-2$ square, so 
$$
\log\det\big(\mathsf{Id}-\frac{1}{4}\mathsf{A}_{\text{Sq}_{L-2\times L-2}}\big)=\prod_{1\leq i,j\leq L-2}1-\frac{1}{2} \cos \left(\frac{\pi  i}{L-1}\right)-\frac{1}{2} \cos \left(\frac{\pi  j}{L-1}\right).
$$
The asymptotic behavior of this when $L\to\infty$ is known. Denote $\mathsf{L}_{\text{Sq}_{L\times L}}$ the Laplacian of the $L\times L $ square, then 
$$
\det\big(\mathsf{Id}-\frac{1}{4}\mathsf{A}_{\text{Sq}_{L-2\times L-2}}\big)=\frac{4^{-(L-2)^2}}{(L-1)^2}\det\,\!\!_{0}(\mathsf{L}_{\text{Sq}_{L-1\times L-1}}),
$$
where $\det_0$ is the product of the non-zero eigenvalues. It follows \cite[Section 6]{Kenyon2000} \cite{Duplantier1988, pozrikidis2014introduction},
\begin{align*}
\log\det\big(\mathsf{Id}-\frac{1}{4}\mathsf{A}_{\text{Sq}_{L-2\times L-2}}\big)&=\left(\frac{4C}{\pi}-\log(4)\right)\mathcal{A}(\text{Sq}_{L-1\times L-1})+\frac{1}{2}\log(4\sqrt{2}-4)\mathcal{P}(\text{Sq}_{L-1\times L-1})+O(\log L),
\end{align*}

where $C$ is Catalan's constant, $\mathcal{A}(\text{Sq}_{L-1\times L-1})=(L-1)^2$ and $\mathcal{P}(\text{Sq}_{L-1\times L-1})=4(L-1)$ are the area and perimeter of $\text{Sq}_{L-1\times L-1}$, respectively.
The area and perimeter are quantities that are invariant under conformal transformations of the lattice under the continuum limit, leading to the conformal invariance of SLE measures.
Unfortunately, neither of the other two determinants in Eq.~(\ref{Fpres1}) make any sense under the limit where $N:=|G_N|\to\infty$ where $G_N\to G$ becomes infinite. The behavior of their ratio is also very difficult to control. Instead of attempting to do so exactly, as Theorem~\ref{InfiniteSieve} achieved, we may only seek an asymptotic description of these determinants as $N\to\infty$. 
To do so, observe that the two ill-defined determinants' contents can be transformed so as to make graph Laplacians appear. Indeed, 
$$
\det(\mathsf{Id}-\Lambda^{-1}\mathsf{A}_{\text{out}(G_N\backslash p)})=\Lambda^{-N} \det(\mathsf{Id}\Lambda-\mathsf{A}_{\text{out}(G_N\backslash p)}),
$$
which looks like a graph Laplacian, except that all vertices are effectively given a degree $\Lambda$. But vertices of $\mathcal{N}(p)$ have a strictly smaller degree on $G_N\backslash p$, by virtue of the removal of the vertices visited by $p$. Let $\mathsf{D}_p$ be a diagonal matrix accounting for this, that is, for any vertex $i$, define $\mathsf{D}_p$ through
$$
\Lambda= \deg(i_{\text{out}(G_N\backslash p)})+(\mathsf{D}_p)_{i,i},
$$
where $\deg(i_{\text{out}(G_N\backslash p)})$ designates the degree of vertex $i$ on the subgraph of $G_N$ induced by the vertex set $\text{out}(G_N\backslash p)$. By construction, $(\mathsf{D}_p)_{i,i}=0$ whenever $i\notin \mathcal{N}(p)$, i.e. $\mathsf{D}_p$ is a "small-rank and localized" correction, the extent of which does not increase as the lattice size is increased. 
This leads to
\begin{equation}\label{DetAout}
\det(\mathsf{Id}-\Lambda^{-1}\mathsf{A}_{\text{out}(G_N\backslash p)})=\Lambda^{-|\text{out}(G\backslash p)|}\det\,\!\!_{0}(\mathsf{L}_{\text{out}(G\backslash p)} + \mathsf{D}_p),
\end{equation}
and similarly, $\det(\mathsf{Id}-\frac{1}{\Lambda}\mathsf{A}_{G\backslash \bullet})=\Lambda^{|G_N|-1}\det(\mathsf{L}_{G\backslash \bullet}+\mathsf{D}_\bullet)$. 
The determinant of Eq.~(\ref{DetAout}) can be expanded around the dominant $\mathsf{L}_{\text{out}(G\backslash p)}$. Here unfortunately we could not complete that step of the proof by showing that the contribution from $\mathsf{D}_p$ was $O(\log \ell(p) )$ in $\log \det(\mathsf{Id}-\Lambda^{-1}\mathsf{A}_{\text{out}(G_N\backslash p)})$, where $\ell(p)$ is the length of $p$. If we accept that this true we would have

\begin{align*}
    \log(F_p)\stackrel{?}{=} &~(-\ell(p)+|\text{in}(G_N\backslash p)|+\mathcal{P}(p)+1) \log \Lambda +\log\det\,\!\!_{0}(\mathsf{L}_{G_N\backslash p})-\log\det\,\!\!_{0}(\mathsf{L}_{G_N\backslash \bullet} )\\
&\hspace{10mm}+\log\det(\mathsf{Id}-\Lambda^{-1}\mathsf{A}_{\text{in}(G_N\backslash p)})+O(\log \ell(p)).
\end{align*}
Here we used $|\text{out}(G\backslash p)|=|G_N|-|\text{in}(G\backslash p)|-\mathcal{P}(p)$.
The two terms involving graph Laplacians are available asymptotically thanks to the remarkable results of Finski \cite{Finski2022}, Greenblatt \cite{Greenblatt2023} and \cite{Kenyon2000}. For example, on the square lattice and in the case of the $L\times L$ square, we have \cite[Eq.~(1.4) and Theorem 1.2]{Finski2022},
\begin{align*}
\log\det\,\!\!_{0}(\mathsf{L}_{G_N\backslash \text{Sq}_{L\times L}})-\log\det\,\!\!_{0}(\mathsf{L}_{G_N\backslash \bullet}) =& \frac{4C}{\pi}\times-\mathcal{A}(\text{Sq}_{L-1\times L-1})+\frac{1}{2}\log(\sqrt{2}-1)\, \mathcal{P}(\text{Sq}_{L\times L}) +O(\log L).
\end{align*}

Since $|\text{in}(G\backslash p)|=(L-1)^2$ and $\mathcal{P}(\text{Sq}_{L\times L})=4L$,
putting everything together yields
\begin{equation}\label{logF}
\log(F_{\text{Sq}_{L\times L}}) \stackrel{?}{=} -4L \log(4) + 4L \log(4\sqrt{2}-4)+O(\log L).
\end{equation}
Note how all terms proportional to $L^2$ stemming from areas have canceled each other out. It cannot be otherwise: if something proportional to $L^\alpha$ with $\alpha>1$ were to remain, then $F_{\text{Sq}_{L\times L}}$ would diverge as $L\to \infty$, in sharp contradiction with the basic combinatorial requirement of Eq.~(\ref{SumOne}). Overall, Eq.~(\ref{logF}) allows us to conjecture that
$$
F_{\text{Sq}_{L\times L}} \stackrel{?}{=} (\sqrt{2}-1)^{4L} + O(\text{polynom}(L)).
$$
This prediction is borne out by the numerical results, see Fig.~\ref{FpCarre}.

The analytical reasoning for the $-3/5$ exponent appearing the sums $S(\ell)$ is much weaker. If indeed localized corrections can be neglected, then the sum of $F_p$ values could possibly be estimated asymptotically using asymptotic determinants of discrete Laplacians. The problem may then be related to that solved by Kenyon \cite{Kenyon2000} as one of the authors of this work had adventurously proposed in \cite{GISCARD2021}.

\section{Generalization to other lattices}
The computational procedure described in the case of the square lattice in Sections~\ref{Cmatrix} and \ref{AlgoSec} extends to other lattices, provided a suitable recursion formula can be found on the entries of the $\mathsf{C}$ matrix. Indeed, while explicit formulas are known for $\mathsf{C}$ via its relation to lattices' Green's functions in all cases \cite{Atkinson1999, Cserti2000, Cserti2011, Guttmann_2010}, these formulas involve integrals of complicated functions and, as such, are ill-suited to fast, large-scale numerical computations requiring high numerical precision. Instead, it is much easier to implement a direct reconstruction of $\mathsf{C}$ exploiting lattice relations \cite{Klein2002ResistanceDistanceSR} and known values for some family of its entries (in the spirit of Proposition~\ref{RecSquare} and Fig.~\ref{Direction_of_computations} for the square lattice). In this respect the triangular lattice plays a special role since many relations are known relating entries of its resistance matrix to entries of the resistance matrices of other lattices including the Kagomé and dice ones \cite{Cserti2011}. Unfortunately, as noted by the author of \cite{Cserti2000}, while an explicit integral formula is known giving all resistances on the triangular lattice, no recurrence relation is known between these values. Consequently, these cannot be efficiently reconstructed by a computational procedure similar to that presented here for the square lattice. We remove this roadblock by not only providing such a recurrence but also solving it explicitly thus obtaining a novel expression for the resistance on the triangular lattice. We also solve the open problem of determining the asymptotic behavior of the resistance between two points of this lattice when their distance $d$ grows to infinity $d\to\infty$. This results not only answer existing question on the infinite resistor triangular lattice but they also establish the validity of our computational approach for the calculation of fractions $F_p$ to a vast array of lattices.

We first recall the analytical results concerning the resistance values on the triangular lattice:
\begin{figure}
\begin{center}
\def\xsize{4}
\def\ysize{4}

\begin{tikzpicture}[line width=0.1,x=1.4cm,y=1.4cm]
\begin{scope}
\clip(0,0) rectangle (\xsize,{\ysize*sqrt(3)/2});
\foreach \y in {0,...,\ysize}\draw[color=lightgray](0,{sqrt(3)*\y/2})  -- ++ (0:\xsize);
\foreach \y in {0,2,...,\ysize}{
\draw[color=lightgray](0,{sqrt(3)*\y/2})  -- ++ (60:2*\ysize);
\draw[color=lightgray](0,{sqrt(3)*\y/2})  -- ++ (-60:2*\ysize);
}
\foreach \x in {1,...,\xsize}{
\draw[color=lightgray](\x,0)  -- ++ (60:2*\ysize);
\draw[color=lightgray](\x,{\ysize*sqrt(3)/2})  -- ++ (-60:2*\ysize);
}
\end{scope}

\filldraw[fill=white](2,{sqrt(3)}) coordinate (O) circle(0.1);
\filldraw[fill=black](2.5,{1.5*sqrt(3)}) coordinate (A) circle(0.06);
\filldraw[fill=black](1.5,{1.5*sqrt(3)}) coordinate (B) circle(0.06);
\filldraw[fill=black](1,{sqrt(3)}) coordinate (C) circle(0.06);
\filldraw[fill=black](1.5,{0.5*sqrt(3)}) coordinate (D) circle(0.06);
\filldraw[fill=black](2.5,{0.5*sqrt(3)}) coordinate (E) circle(0.06);
\filldraw[fill=black](3,{sqrt(3)}) coordinate (F) circle(0.06);

\draw[line width=1,->](O) -- node[midway,left]{$v$} ++ (120:0.92);
\draw[line width=1,->](O) -- node[midway,above]{$u$} ++ (0:0.92);

\filldraw[fill=white](2,{sqrt(3)}) coordinate (O) circle(0.1);

\draw[fill=white](O) node[xshift=0pt,yshift=-10pt]{\footnotesize$(i,j)$};
\draw[fill=white](A) node[xshift=24pt,yshift=5pt]{\footnotesize$(i\!+\!1,j\!+\!1)$};
\draw[fill=white](B) node[xshift=-24pt,yshift=5pt]{\footnotesize$(i,j\!+\!1)$};
\draw[fill=white](C) node[xshift=-20pt,yshift=5pt]{\footnotesize$(i\!-\!1,j)$};
\draw[fill=white](D) node[xshift=-24pt,yshift=-5pt]{\footnotesize$(i\!-\!1,j\!-\!1)$};
\draw[fill=white](E) node[xshift=24pt,yshift=-5pt]{\footnotesize$(i\!+\!1,j\!-\!1)$};
\draw[fill=white](F) node[xshift=20pt,yshift=5pt]{\footnotesize$(i\!+\!1,j)$};

\end{tikzpicture}
\end{center}
\caption{\label{CoordTri}Coordinate system on triangular lattice.}
\end{figure}
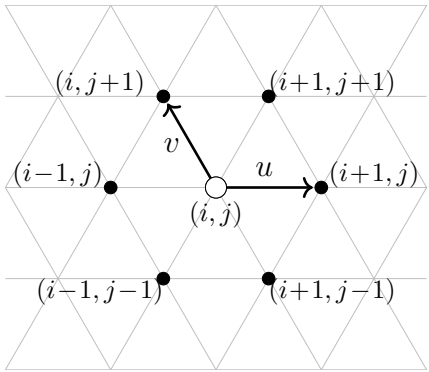

\begin{prop}[\cite{Cserti2000}]
Let $R(n,m)$ be the resistance from the origin $O$ to the point with coordinates $(n,m)$ on the triangular lattice (see Fig.~\ref{CoordTri}). Then
$$
R(n,m)=\frac{1}{\pi}\int_0^{\pi/2}\frac{1-e^{-|n-m|s}\cos\big((n+m)s\big)}{\sinh(s)\cos(x)}dx,
$$
where $\cosh(s)=2\sec(x)-\cos(x)$. In particular for $n=m$ this simplifies to
$$
R(n,n)= \frac{2}{\pi}\int_0^{\pi/2}\frac{\sin(n x)^2}{\sin(x) \sqrt{4-\cos(x)^2}}dx.
$$
\end{prop}

Of central importance in the computational procedure for the evaluation of $\mathsf{C}$ is the determination of these $R(n,n)$, ideally without integrals, since from these all $R(n,m)$ can be recursively calculated using sum rules and lattice symmetries \cite{Klein2002ResistanceDistanceSR}. The following Proposition identifies a recurrence (conjectured to exist by \cite{Cserti2000}) to bypass the integral formulation altogether; identifies its explicit solution in terms of hypergeometric functions and proves the conjectured \cite{Cserti2000} logarithmic divergence of $R(n,n)$:  

\begin{prop}\label{Prop_induction_sur_r}
Let $r_n:=R(n,n)$ and define $r_{-1}=0$ for convenience. Then, on the triangular lattice we have $r_0=0$, $r_1=1/3$ and for $n\geq 2$,
\begin{align}\label{induction_sur_r}
r_n&=\frac{15n-22}{n-1}r_{n-1}-\frac{15n-23}{n-1}r_{n-2}+\frac{n-2}{n-1}r_{n-3}-\frac{4\sqrt{3}}{\pi(n-1)}.
\end{align}
It follows that $r_n$ is available exactly from a finite number of sums, ratios and products of rational numbers and a single multiplication by the irrational $\sqrt{3}/\pi$.
The solution of this is, for $n\geq 0$,
\begin{align*}
r_n=&\frac{n}{3}H(n)-\frac{4 \sqrt{3}}{\pi}\sum_{m=1}^n (n-m) H(n-m)H(m),
\end{align*}
where $H(n):=\,_3F_2\left(\frac{1}{2},-n+1,n+1;1,\frac{3}{2};-3\right)\in\mathbb{Q}$. Here $\,_3F_2$ designates a generalized hypergeometric function \cite{bailey1935generalized}.
Finally we have, asymptotically as $n\to+\infty$,
$$
r_n =\frac{1}{\sqrt{3}\,\pi}\log(n)+\frac{ \gamma +\log (2\sqrt{3})}{ \sqrt{3}\, \pi }+O(1/n),
$$
where $\gamma$ is the Euler–Mascheroni constant.
\end{prop}

\begin{proof}
First we prove the recurrence.    Observe that, since $2\cos^2(x)-\cos(2x)=1$, we have $r_n=\frac{2\sqrt{2}}{\pi}\int_0^{\pi/2}\frac{\sin(n x)^2}{\sin(x) \sqrt{7-\cos(2x)}}dx$. Let $f_n(x):=\frac{\sin(n x)^2}{\sin(x) \sqrt{7-\cos(2x)}}$. We claim
    \begin{align*}
        f_n(x) &= \frac{15 n-22}{n-1}f_{n-1}(x)-\frac{15 n-23}{n-1}f_{n-2}(x)+\frac{n-2}{n-1}f_{n-3}(x)\\
        &\hspace{10mm}+\frac{1}{n-1} \frac{d}{dx}\left\{\sqrt{7-\cos (2 x)}\, \cos \big((2 n-3) x\big)\right\}.
    \end{align*}
    Actually, it suffices to show that 
    \begin{align*}
    (n-1)\sin(nx)^2&=(15n-22)\sin((n-1)x)^2-(15n-23)\sin((n-2)x)^2+(n-2)\sin((n-3)x)^2\\
    &+\cos((2n-3)x)\sin(2x)\sin(x)-(2n-3)(7-\cos(2x))\sin(x)\sin((2n-3)x).
    \end{align*}
    However, since $2\sin(\theta)^2=1-\cos(2\theta)$, it is equivalent to 
    \begin{align*}
        (n-1)\cos(2nx)&=(15n-22)\cos(2(n-1)x)-(15n-23)\cos(2(n-2)x)+(n-2)\cos(2(n-3)x)\\
            &+\cos((2n-3)x)\sin(2x)\sin(x)-(2n-3)(7-\cos(2x))\sin(x)\sin((2n-3)x).
    \end{align*}
    This relation is clear using the Product to Sum identities ($2\cos(\theta)\cos(\mu)=\cos(\theta-\mu)+\cos(\theta+\mu)$ and $2\sin(\theta)\sin(\mu)=\cos(\theta-\mu)+\cos(\theta+\mu)$) on the the RHS of the relation.\\

Second we establish the solution. We begin by considering the homogeneous recurrence equation in $h_n$ for $n\geq 2$,
\begin{equation}\label{recHomogen}
h_n=\frac{15n-22}{n-1}h_{n-1}-\frac{15n-23}{n-1}h_{n-2}+\frac{n-2}{n-1}h_{n-3},
\end{equation}
and with  $h_{-1}=0$, $h_0=0$ and $h_1=1/3$.
Define for $n,k,\in\mathbb{N}$, 
$$
c(n,k):=\frac{1}{(2 k+1) (k!)^2}\frac{(n+k)!}{(n-k-1)!},
$$ and observe that for all integers $n,k\geq 2$,
\begin{align}\label{relC}
&c(n,k)=\frac{3 n-4}{n-1}c(n-1,k)-\frac{4 n-6}{n-1}c(n-1,k-1)\\
&\hspace{15mm}-\frac{3 n-5 }{n-1}c(n-2,k)+\frac{4 n-6 }{n-1}c(n-2,k-1)+\frac{n-2}{n-1}c(n-3,k),\nonumber
\end{align}
noting that $c(-1,k)=0$ for all $k\geq 1$.
Since furthermore
$$
n\, Q_n(z):=\sum_{k=0}^n c(n,k) z^k=n \, _3F_2\left(\frac{1}{2},1-n,n+1;1,\frac{3}{2};z\right),
$$
then Eq.~(\ref{relC}) implies 
\begin{align*}
Q_n(z) &=\left(-\frac{4 n-6}{n-1}z+\frac{3 n-4}{n-1}\right) (n-1)Q_{n-1}(z) \\
&\hspace{5mm}+\left(\frac{4 n-6}{n-1} z-\frac{3 n-5}{n-1}\right)(n-2)Q_{n-2}(z)+\frac{n-2}{n-1} (n-3)Q_{n-3}(z).
\end{align*}
This becomes Eq.~(\ref{recHomogen}) for $z=-3$, showing that the solution of the homogeneous recurrence relation Eq.~(\ref{recHomogen}) is $h_n=n\,Q_n(-3)$ and $H(n)=Q_n(-3)$.

The solution to the inhomogeneous equation has a similar origin. Eq.~(\ref{relC}) implies that $P_n(z):=\sum_{m=1}^n (n-m)Q_{n-m}(z) Q_m(z)$ satisfies
$$
P_n(z) =\left(-\frac{4 n-6}{n-1}z+\frac{3 n-4}{n-1}\right) P_{n-1}(z) +\left(\frac{4 n-6}{n-1} z-\frac{3 n-5}{n-1}\right)P_{n-2}(z)+\frac{n-2}{n-1} P_{n-3}(z) +\frac{1}{n-1},
$$
implying that $-(4\sqrt{3}/\pi)P_n(-3)=-(4\sqrt{3}/\pi)\sum_{m}(n-m)H(n-m)H(m)$ is a particular solution to the inhomogeneous equation.

Finally, we  prove that $r_n$ diverges logarithmically as $n\to+\infty$. Observe that 
$$
\frac{1}{\sin(x) \sqrt{4-\cos(x)^2}}=\frac{1}{\sqrt{3} x}+\frac{4 x^3}{45 \sqrt{3}}-\frac{2 x^5}{63 \sqrt{3}}+O\left(x^6\right),
$$
thus
$$
\frac{2}{\pi}\int_0^{\pi/2}\frac{1}{\sqrt{3} x}\sin(n x)^2 dx\leq r_n\leq\frac{2}{\pi}\int_0^{\pi/2}\left(\frac{1}{\sqrt{3} x}+\frac{4 x^3}{45 \sqrt{3}}\right)\sin(n x)^2 dx.
$$
Evaluating both integrals and multiplying everything by $\sqrt{3} \pi$ for ease of presentation we get
\begin{align*}
&-\text{Ci}(n \pi )+\gamma +\log (\pi )+\log (n)\leq \sqrt{3} \pi\, r_n\leq-\text{Ci}(n \pi )+\gamma +\log (\pi )+\log (n)\\&\hspace{80mm}+\frac{24 \left((-1)^n-1\right)}{720 n^4}-\frac{12 \pi ^2 (-1)^n}{720 n^2}+\frac{\pi ^4}{720}.
\end{align*}
Here $\gamma$ is the Euler–Mascheroni constant and $\text{Ci}(.)$ is the cos-integral function for which we have $\lim_{n\to+\infty}\text{Ci}(n \pi)= 0$. It follows from the above that asymptotically as $n\to+\infty$, 
$
r_n\sim \log(n)/(\sqrt{3}\pi)+c+O(1/n).$
The value of $c$ is determined by the observation that for $q\geq 0$,
$$
\lim_{n\to\infty}\int_0^{\pi/2} x^q \sin(nx)^2 dx =\frac{1}{2}\times \frac{(\pi/2) ^{q+1}}{q+1}.
$$
Given that $\int x^q dx = x^{q+1}/(q+1)$ and since 
$$
F(x):=\int \left(\frac{1}{\sin(x) \sqrt{4-\cos(x)^2}}-\frac{1}{\sqrt{3} x}\right) dx = \frac{-\log (x)}{\sqrt{3}}-\frac{1}{\sqrt{3}}\tanh ^{-1}\left(\frac{\sqrt{3} \cos (x)}{\sqrt{4-\cos ( x)^2}}\right),
$$
then we have
$$
\lim_{n\to\infty}\int_0^{\pi/2} \left(\frac{1}{\sin(x) \sqrt{4-\cos(x)^2}}-\frac{1}{\sqrt{3} x}\right) \sin(nx)^2 dx =\frac{1}{2} \big(F(\pi/2)-F(0)\big),
$$
where we understand $F(0)$ as the limit $F(0):=\lim_{x\to 0} F(x) = -\log (\sqrt{3})/ \sqrt{3}$ while $F(\pi/2)=-\log \left(\pi/2\right)/\sqrt{3}$. Gathering everything we get
$$
c= \frac{\gamma+\log(\pi)}{\sqrt{3}\pi}+\frac{2}{\pi}\times \frac{1}{2}\times \left(-\frac{\log \left(\pi/2\right)}{\sqrt{3}}+\frac{\log (\sqrt{3})}{ \sqrt{3}}\right),
$$
which simplifies to the promised result.

\end{proof}

\section*{Acknowledgements}
 P.-L. G. and Y. H. are supported by the French National Research Agency ANR-19-CE40-0006 project ALgebraic COmbinatorics of Hikes On Lattices (\textsc{Alcohol}). Y. H. is further supported by an Université du Littoral Côte d'Opale postdoctoral grant.

\bibliographystyle{plain}  
\bibliography{references}

\end{document}